\documentclass[a4paper,11pt]{amsart}

\usepackage{amsmath,amssymb,amsthm,latexsym,amscd,mathrsfs}
\usepackage{indentfirst}
\usepackage{stmaryrd}
\usepackage{graphicx}
\usepackage{extarrows}
\usepackage{multirow}
\usepackage{latexsym}
\usepackage{amsfonts}
\usepackage{color}
\usepackage{pictexwd,dcpic}
\usepackage{graphicx}
\usepackage{psfrag}
\usepackage{hyperref}
\usepackage{comment}
\usepackage{enumerate}

\numberwithin{equation}{section} \theoremstyle{plain}
\newtheorem{thm}{Theorem}[section]

\newtheorem{prop}[thm]{Proposition}
\newtheorem{lem}[thm]{Lemma}
\newtheorem{cor}[thm]{Corollary}
\newtheorem{defn}[thm]{Definition}

\newtheorem{exmp}[thm]{Example}

\newtheorem{rem}[thm]{Remark}

\newtheorem*{ques0}{Question}

\newtheorem{ack}{Acknowledgements}   

\makeatletter

\def\<{\langle}
\def\>{\rangle}
\def\({\left(}
\def\){\right)}
\def\[{\left[}
\def\]{\right]}

\makeatother

\title{Austere Matrices, Austere Submanifolds and Dupin Hypersurfaces}

\author[J.Q. Ge]{Jianquan Ge}
\address{School of Mathematical Sciences, Laboratory of Mathematics and Complex Systems, Beijing Normal University, Beijing 100875, P.R. CHINA.}
\email{jqge@bnu.edu.cn}

\author[Y. Zhou]{Yi Zhou$^{*}$}
\address{Beijing International Center for Mathematical Research, Peking University,
Beijing 100871, P.R. CHINA.}
\address{School of Mathematical Sciences, Laboratory of Mathematics and Complex Systems, Beijing Normal University, Beijing 100875, P.R. CHINA.}
\email{yizhou@bicmr.pku.edu.cn, zhou\_yi@mail.bnu.edu.cn}

\subjclass[2010]{53C42, 53C40, 53A07.}
\keywords{austere matrix, austere submanifold, Dupin hypersurface.}
\thanks {$^{*}$ the corresponding author.}
\thanks{J. Q. Ge is partially supported by NSFC (No. 12171037, 12571049) and the Fundamental Research Funds for the Central Universities.}
\thanks{Y. Zhou is partially supported by  NSFC (No. 12171037, 12271040), China Postdoctoral Science Foundation (No. BX20230018)
and National Key R$\&$D Program of China 2020YFA0712800.}

\begin{document}
\maketitle
\begin{abstract}
Motivated by Bryant's research on austere subspaces and
Cartan's isoparametric hypersurfaces with 3 distinct principal curvatures,
we construct three families of austere submanifolds with flat normal bundle in unit spheres.
From these examples we find three irreducible proper Dupin hypersurfaces
with 5 distinct principal curvatures of different multiplicities.
Thus, we give a negative answer to an open question raised by Thorbergsson in 2000
which is instructive for the local classification of proper Dupin hypersurfaces.
Moreover, as an application, we obtain an upper bound estimate for the dimension of austere subspaces.
\end{abstract}

{\center{{\emph{\quad \quad \quad \quad \quad \quad Dedicated to Professor Chia-Kuei Peng on his 80th birthday}}}}

\section{Introduction}
A submanifold of a Riemannian manifold is said to be austere if
all symmetric polynomials of odd degree in the principal curvatures
with respect to each normal vector vanish, or equivalently,
the principal curvatures with respect to every normal vector is of the form
$$\lambda_1, -\lambda_1, \lambda_2, -\lambda_2, \cdots, \lambda_p, -\lambda_p, 0, \cdots, 0.$$
This concept was introduced by Harvey-Lawson \cite{HL} in their constructions
of special Lagrangian submanifolds in $\mathbb{C}^N$.
It is easy to see that austerity implies minimality,
and these two conditions are equivalent for curves and surfaces.

Bryant \cite{Bryant} developed some algebra to
classify possible second fundamental forms of austere submanifolds.
By fixing an orthonormal basis,
he regarded the vector space of quadratic forms on an $n$-dimensional real vector space
as the vector space $\mathcal{S}_n$ of $n\times n$ real symmetric matrices.
For convenience, we say that an $n\times n$ real symmetric matrix $A$ is austere if $$\operatorname{tr}A^{2k+1}=0\ \mbox{for each}\ 0\leq k\leq \[\frac{n-1}2\],$$
or equivalently, all symmetric polynomials of odd degree in the eigenvalues of $A$ vanish.
Then the problem is to classify linear subspaces $\mathcal{Q}\subset\mathcal{S}_n$
consisting of austere matrices, called austere subspaces.
Through his insightful observation, Bryant constructed several maximal austere subspaces (see Table \ref{max_austere})
and completed the classification for the cases of $n=2,3,4$ (see Table \ref{classaust}), up to the conjugate action of $\operatorname{O}(n)$ on $\mathcal{S}_{n}$ given by $P\cdot A=P^tAP$.
\begin{table}[ht!]\small
\caption{Bryant's constructions of maximal austere subspaces}\label{max_austere}
\centering
\begin{tabular}{|c|c|c|}
\hline
Order & Maximal austere subspaces & Dimension\\
\hline
\ & $\left\{\begin{pmatrix}
m_1 & m_2  \\
m_2 & -m_1 \\
\end{pmatrix}: m_1, m_2\in\mathcal{S}_{p+1}\right\}$ & $(p+1)(p+2)$\\
\cline{2-3}
$n=2p+2$  & $\left\{\begin{pmatrix}
0 & m  \\
m^t & 0 \\
\end{pmatrix}: m\in M(k, n-k, \mathbb{R})\right\},\ 1\leq k\leq p$ & $k(n-k)$\\
\cline{2-3}
\  & $\left\{\begin{pmatrix}
\lambda I_{p+1} & m  \\
m^t & -\lambda I_{p+1} \\
\end{pmatrix}: m\in M(p+1, \mathbb{R}),\ \lambda\in\mathbb{R}\right\}$ & $(p+1)^2+1$\\
\hline
\ & $\left\{\begin{pmatrix}
m_1 & m_2 & 0 \\
m_2 & -m_1 & 0 \\
0 & 0 & 0 \\
\end{pmatrix}: m_1, m_2\in\mathcal{S}_{p}\right\}$ & $p(p+1)$\\
\cline{2-3}
$n=2p+1$ & $\left\{\begin{pmatrix}
0 & m  \\
m^t & 0 \\
\end{pmatrix}: m\in M(k, n-k, \mathbb{R})\right\},\ 1\leq k\leq p$ & $k(n-k)$\\
\cline{2-3}
\ & $\left\{\begin{pmatrix}
\lambda I_{p} & m & 0 \\
m^t & -\lambda I_{p} & 0 \\
0 & 0 & 0 \\
\end{pmatrix}: m\in M(p, \mathbb{R}),\ \lambda\in\mathbb{R}\right\}$ & $p^2+1$\\
\hline
\end{tabular}
\end{table}
\begin{table}[ht!]\small
\caption{Bryant's classifications for $n=2,3,4$}\label{classaust}
\centering
\begin{tabular}{|c|c|c|}
\hline
Order & Maximal austere subspaces & Dimension\\
\hline
$n=2$ & $\left\{\begin{pmatrix}
a & b  \\
b & -a \\
\end{pmatrix}: a, b\in\mathbb{R}\right\}$ & $2$\\
\hline
$n=3$ & $\left\{\begin{pmatrix}
a & b & 0 \\
b & -a & 0 \\
0 & 0 & 0 \\
\end{pmatrix}: a, b\in\mathbb{R}\right\}$ & $2$\\
\cline{2-3}
\ & $\left\{\begin{pmatrix}
\,0 & 0 & a\, \\
\,0 & 0 & b\, \\
\,a & b & 0\, \\
\end{pmatrix}: a, b\in\mathbb{R}\right\}$ & $2$\\
\hline
$n=4$ & $\left\{\begin{pmatrix}
m_1 & m_2  \\
m_2 & -m_1 \\
\end{pmatrix}: m_1, m_2\in\mathcal{S}_{2}\right\}$ & $6$\\
\cline{2-3}
\  & $\left\{\begin{pmatrix}
\lambda I_{2} & m  \\
m^t & -\lambda I_{2} \\
\end{pmatrix}: m\in M(2, \mathbb{R}),\ \lambda\in\mathbb{R}\right\}$ & $5$\\
\cline{2-3}
\  & $\left\{\begin{pmatrix}
0 & x_1 & x_2 & x_3  \\
x_1 & 0 & \lambda_3x_3 & \lambda_2x_2 \\
x_2 & \lambda_3x_3 & 0 & \lambda_1x_1 \\
x_3 & \lambda_2x_2 & \lambda_1x_1 & 0 \\
\end{pmatrix}: x_1, x_2, x_3\in\mathbb{R}\right\}$, where the & $3$ \\
\  & constants $\lambda_1\geq\lambda_2\geq0\geq\lambda_3$ satisfy $\lambda_1\lambda_2\lambda_3+\lambda_1+\lambda_2+\lambda_3=0.$ & \ \\
\hline
\end{tabular}
\end{table}

A large family of austere submanifolds come naturally from the focal submanifolds of isoparametric hypersurfaces in unit spheres which are hypersurfaces of constant principal curvatures initially studied since 1938 by Cartan and classified recently after many efforts (cf. \cite{CCJ07, Chi11, Chi13, Chi16, Chi19, DN85, Miy13, Miy16}, etc.).
In Cartan's papers \cite{Cartan1, Cartan2}, he proved that an isoparametric family of hypersurfaces
in unit spheres with 3 distinct constant principal curvatures of equal multiplicity $m$ is given by
the regular level sets of the Cartan-M\"{u}nzner polynomials
\begin{equation}\label{Cartan-poly}
\begin{aligned}
F_C(x, y, X, Y, Z)=&x^3-3xy^2+\frac32x(X\overline{X}+Y\overline{Y}-2Z\overline{Z})\\
&+\frac{3\sqrt{3}}2y(X\overline{X}-Y\overline{Y})+\frac{3\sqrt{3}}2(XYZ+\overline{Z}\overline{Y}\overline{X}),
\end{aligned}
\end{equation}
where $x, y\in\mathbb{R}$ and $X, Y, Z\in\mathbb{R}, \mathbb{C}, \mathbb{H}, \mathbb{O}$ for
$m=1, 2, 4, 8$, respectively.
Cartan also showed that all these isoparametric hypersurfaces are homogeneous,
that is, they are principal orbits in the unit spheres under the group actions of
$\operatorname{SO}(3)$, $\operatorname{SU}(3)$, $\operatorname{Sp}(3)$ or the exceptional Lie group $F_4$.
In \cite{CR85, CR15}, Cecil-Ryan pointed out that the former three actions can be presented as conjugate actions
on the linear space of $3\times3$ traceless real, complex Hermitian, or quaternionic Hermitian matrices.
For example, let $E^5$ be the 5-dimensional linear space of $3\times3$ real symmetric matrices with trace zero,
and let $S^4$ be the unit sphere in $E^5$ with respect to the Frobenius norm.
Then the Cartan-M\"{u}nzner polynomial of (\ref{Cartan-poly}) is nothing but the determinant function (cf. \cite{GT22}), and the principal orbits of the isometric $\operatorname{SO}(3)$-action
$$\operatorname{SO}(3)\times S^4\rightarrow S^4,\ (P, A)\mapsto P^tAP,$$
give exactly the isoparametric family of hypersurfaces with 3 distinct principal curvatures of multiplicity $m=1$.
In particular, the unique minimal hypersurface in this isoparametric family is the orbit of
$$\frac1{\sqrt2}\begin{pmatrix}
1 & 0 & 0 \\
0 & -1 & 0 \\
0 & 0 & 0 \\
\end{pmatrix}.$$
Let $\mathcal{A}_n$ denote the set of $n\times n$ austere matrices with Frobenius norm $1$.
We observe that the preceding  minimal isoparametric hypersurface is an austere submanifold in $S^4$ and is exactly the set $\mathcal{A}_3$.
Thus the intersection of $S^4$ and every austere subspace of $\mathcal{S}_3$
is a totally geodesic sphere contained in $\mathcal{A}_3$.

Based on these observations, we want to generalize the submanifold structure of $\mathcal{A}_3$ to (a subset of) $\mathcal{A}_n$, as well as the complex and quaternionic versions.
As a result, we obtain three families of new examples of austere submanifolds in unit spheres (see Theorem \ref{Bn austere}). Among these three families, the $n=3$ case gives rise to Cartan's three isoparametric hypersurfaces with $g=3$ and multiplicity $m=1,2,4$ respectively, while the ``regular" part of the $n=4$ case turns out to be also very interesting examples of Dupin hypersurfaces (see in the follows and Theorems \ref{B4_pD}, \ref{B4_irr}). Besides, the ``singular" part of the $n=4$ case also provides us three new examples of compact austere submanifolds (see Proposition \ref{Cn austere}).
As an application, we obtain an upper bound for the dimension of austere subspaces under some assumptions (see Theorem \ref{dimest}).

Recall that an oriented hypersurface $f: M^n \rightarrow\widetilde{M}^{n+1}$ is called a Dupin hypersurface if along each curvature surface, the corresponding principal curvature is constant.
A Dupin hypersurface $M^n$ is called proper Dupin if
the number $g$ of distinct principal curvatures is constant on $M^n$.
In real space forms, proper Dupin hypersurfaces can be seen as a generalization of isoparametric hypersurfaces.
Pinkall \cite{Pinkall} introduced four methods to construct a Dupin hypersurface from a lower-dimensional Dupin hypersurface,
and he showed that there exists a proper Dupin hypersurface for any given number $g$ of distinct principal curvatures and any given multiplicities $m_1, \cdots, m_g$.
For compact proper Dupin hypersurfaces in spheres,
Thorbergsson \cite{Thorbergsson83} proved that the number $g$ of distinct principal curvatures can only be $1, 2, 3, 4$ or $6$, the same restriction as for isoparametric hypersurfaces proven by M\"{u}nzner \cite{Mun} (see a generalization in \cite{Fa17}).

Reducibility is a subtle but important property of Dupin submanifolds.
A Dupin submanifold obtained from lower-dimensional Dupin submanifold
via one of Pinkall's constructions is said to be reducible.
More generally, a Dupin submanifold which is locally Lie equivalent to
such a Dupin submanifold is said to be reducible \cite{Cecil}.
Cecil-Chi-Jensen \cite{CCJ} proved that every compact, connected proper Dupin
hypersurface in the Euclidean space with $g>2$ distinct principal curvatures is irreducible.
In the elegant survey \cite{Thorbergsson00}, Thorbergsson raised the following question:
\begin{ques0}
Is it possible that by assuming irreducibility instead of compactness of a proper Dupin
hypersurface the conclusion $g = 1, 2, 3, 4$ or $6$ can be drawn
as well as the restrictions on the multiplicities in Stolz's theorem \cite{Stolz}?
\end{ques0}

Unlike the case of isoparametric hypersurfaces, the local and global classification of
proper Dupin hypersurfaces are quite different.
Several important progress on local classification of proper Dupin hypersurfaces was made
under the assumption of irreducibility \cite{CCJ, CJ98, CJ00}.
Thus the answer to Thorbergsson's question may guide the research on the local classification of proper Dupin hypersurfaces.

To explore a more appropriate description of reducibility, Dajczer-Florit-Tojeiro \cite{DFT} gave the concept of weak reducibility for proper Dupin hypersurface in Euclidean spaces.
A proper Dupin hypersurface $M^n$ in $\mathbb{R}^{n+1}$ is said to be weakly reducible if it has a principal curvature function $\lambda$ with integrable conullity $E_\lambda^{\bot}$, where $E_\lambda^{\bot}$ is the distribution orthogonal to the principal distribution $E_\lambda$ in the tangent bundle of $M$.
Since this property is invariant under Lie sphere transformations,
the same terminology can be defined equivalently for proper Dupin hypersurfaces in unit spheres. Then they showed that
if a proper Dupin hypersurface $M^n$ in $\mathbb{R}^{n+1}$
is Lie equivalent to a proper Dupin hypersurface with $g+1$ distinct principal curvatures that is obtained from a proper Dupin hypersurface with $g$ distinct principal curvatures by one of Pinkall's constructions, then $M^n$ is weakly reducible.
For simplicity, we say that such a proper Dupin hypersurface is strongly reducible.
In summary, we have the following implications among these ``reducibilities" and their reverses:
$$\mbox{reducible}\Leftarrow\mbox{strongly reducible}\Rightarrow\mbox{weakly reducible},$$
$$\mbox{irreducible}\Rightarrow\mbox{strongly irreducible}\Leftarrow\mbox{weakly irreducible}.$$
As Cecil \cite[p.16]{Cecil08} and \cite[Remark 5.14]{Cecil} showed, there indeed exist reducible but weakly and strongly irreducible examples (which can be obtained by Pinkall's constructions without adding one distinct principal curvature, see also in Propositions 5.4, 5.8, 5.9 of \cite{Cecil}). Therefore, the ``reducibility" in \cite{DFT} should be understood as strong reducibility and their counterexamples for Thorbergsson's question claimed in \cite[p.763]{DFT} are only valid for strongly irreducible, not for irreducible or weakly irreducible case.

During the study of our examples of austere submanifolds from the set $\mathcal{A}_n$ of austere matrices, we find that in the $n=4$ case there are three noncompact irreducible proper Dupin hypersurfaces with 5 distinct principal curvatures of different multiplicities (Theorem \ref{B4_pD} and \ref{B4_irr}). Notice that for the compact case with odd $g$, the multiplicities must be equal.
Thus, we give a negative answer to Thorbergsson's question in full generality.
In addition, Proposition \ref{B4R_wr} illustrates that replacing irreducibility by weak irreducibility in this question does not change the answer. Intuitively, it seems that Thorbergsson's question should be still right in some sense.

The rest of this paper is organized as follows.
In Section \ref{sec-pre}, we give some preliminaries and construct our examples.
In Section \ref{secproperties}, we show the austerity of the submanifolds and some basic properties.
In Section \ref{sec-Dupin}, we verify our counterexamples for Thorbergsson's question by the ``regular" part of the $n=4$ case.
In Section \ref{secC4}, we prove the austerity of the ``singular" part of the $n=4$ case.
At last, in Section \ref{sec-dim}, we prove our upper bound estimate for the dimension of austere subspaces.

\section{Preliminaries}\label{sec-pre}
Let $\mathbb{F}$ be one of $\mathbb{R}, \mathbb{C}$ or $\mathbb{H}$,
and denote by $m:=\operatorname{dim}_{\mathbb{R}}(\mathbb{F})=1,2,4$, respectively.
For each $n\geq2$, let
$$N(n, \mathbb{F}):=\frac12 n(n-1)m+n-1
=\left\{\begin{aligned}
\frac12 n(n+1)-1\ \ \ &\mbox{if}\ \mathbb{F}=\mathbb{R};\\
n^2-1\ \ \ \ \ \ \ \ &\mbox{if}\ \mathbb{F}=\mathbb{C};\\
2n^2-n-1\ \ \ \ \, &\mbox{if}\ \mathbb{F}=\mathbb{H}.
\end{aligned}
\right.$$
Then the set of traceless real, complex or quaternionic Hermitian matrices
$$E(n, \mathbb{F}):=\{A\in M(n, \mathbb{F}): A^*=A, \operatorname{tr}A=0\}$$
is a $N(n, \mathbb{F})$-dimensional real subspace of $M(n, \mathbb{F})$,
where $A^*$ denotes the conjugate transpose of $A$.
The usual Euclidean inner product on $M(n, \mathbb{F})$ is given by
$$\<A, B\>:=\mathfrak{R}(\operatorname{tr}AB^*),$$
where $\mathfrak{R}$ denotes the real part.
For $1\leq i<j\leq n$, let
$$\hat{E}_{ij}:=\frac1{\sqrt{2}}(E_{ij}+E_{ji}),\ \check{E}_{ij}:=\frac1{\sqrt{2}}(E_{ij}-E_{ji}),$$
where $E_{ij}$ is the matrix with $(i,j)$ entry $1$ and all others $0$.
For $2\leq i\leq n$,  let $$\hat{E}_{i}:=\frac1{\sqrt{2}}(E_{11}-E_{ii}).$$
Let $$\mathcal{E}_0:=\{\hat{E}_{i}: 2\leq i\leq n\},
\mathcal{E}_1:=\{\hat{E}_{ij}: 1\leq i<j\leq n\},$$
$$\mathcal{E}_2:=\{\mathbf{i}\check{E}_{ij}: 1\leq i<j\leq n\},
\mathcal{E}_3:=\{\mathbf{j}\check{E}_{ij}: 1\leq i<j\leq n\},
\mathcal{E}_4:=\{\mathbf{k}\check{E}_{ij}: 1\leq i<j\leq n\}.$$
Then it is easy to see that $\mathcal{E}_0\cup\mathcal{E}_1$
is a basis of $E(n, \mathbb{R})$,
$\mathcal{E}_0\cup\mathcal{E}_1\cup\mathcal{E}_2$
is a basis of $E(n, \mathbb{C})$ and
$\mathcal{E}_0\cup\mathcal{E}_1\cup\mathcal{E}_2\cup\mathcal{E}_3\cup\mathcal{E}_4$
is a basis of $E(n, \mathbb{H})$.

For each $k\geq2$, define $$F_k: E(n, \mathbb{F})\rightarrow\mathbb{R},\ A\mapsto\operatorname{tr}A^k.$$
Then for any $X, Y\in T_AE(n, \mathbb{F})\cong E(n, \mathbb{F})$,
$$\begin{aligned}
D(F_k)_A(X)&=\frac{d}{dt}\bigg|_{t=0}F_k(A+tX)=\frac{d}{dt}\bigg|_{t=0}\operatorname{tr}(A+tX)^k\\
&=k\mathfrak{R}(\operatorname{tr}A^{k-1}X)=k\<A^{k-1}, X\>,
\end{aligned}$$
$$\begin{aligned}
D^2(F_k)_A(X,Y)&=\frac{d}{dt}\bigg|_{t=0}(DF_k)_{A+tY}(X)=\frac{d}{dt}\bigg|_{t=0}k\<(A+tY)^{k-1},X\>\\
&=k\sum_{l=0}^{k-2}\<A^{k-2-l}YA^l,X\>=k\sum_{l=0}^{k-2}\<XA^l,A^{k-l-2}Y\>.
\end{aligned}$$
Thus the gradient of $F_k$ on $E(n, \mathbb{F})$ is $$\operatorname{grad}^EF_k\Big|_A=
k\[A^{k-1}-\frac1n\(\operatorname{tr}A^{k-1}\)I_n\]\in T_AE(n, \mathbb{F}).$$
Let $S(n, \mathbb{F})$ be the unit sphere in $E(n, \mathbb{F})$,
and let $f_k:=F_k\big|_{S(n, \mathbb{F})}$ for each $k\geq3$.
Then the gradient of $f_k$ on $S(n, \mathbb{F})$ is $$\operatorname{grad}^Sf_k\Big|_A=
k\[A^{k-1}-\(\operatorname{tr}A^k\)A-\frac1n\(\operatorname{tr}A^{k-1}\)I_n\]\in T_AS(n, \mathbb{F}),$$ and the Hessian is
$$\(\operatorname{Hess}^Sf_k\)_A(X,Y)=D^2(F_k)_A(X,Y)-k\operatorname{tr}A^k\<X,Y\>$$
for any $X, Y\in T_AS(n, \mathbb{F})$.

Define $$\Phi_{n, \mathbb{F}}: S(n, \mathbb{F})\rightarrow\mathbb{R}^p,\
A\mapsto(f_3(A), f_5(A), \cdots, f_{2p+1}(A)),$$
where $p=\[\frac{n-1}2\]$.
Then we have $$\mathcal{A}_n=\left\{A\in S(n, \mathbb{R}):
f_{2k+1}(A)=0\ \mbox{for each}\ 0\leq k\leq \[\frac{n-1}2\]\right\}
=\Phi_{n, \mathbb{R}}^{-1}(0).$$
Let $\mathcal{B}_{n, \mathbb{F}}$ and $\mathcal{C}_{n, \mathbb{F}}$ denote
the subsets of regular points and critical points of $\Phi_{n, \mathbb{F}}$
in $\Phi_{n, \mathbb{F}}^{-1}(0)$, respectively.
Note that for each smooth vector field $X$ on $S(n, \mathbb{F})$,
$$d\Phi_{n, \mathbb{F}}(X)=
\(\<\operatorname{grad}^Sf_3, X\>, \cdots, \<\operatorname{grad}^Sf_{2p+1}, X\>\).$$
Let $$G_{n, \mathbb{F}}: S(n, \mathbb{F})\rightarrow\mathbb{R},\
A\mapsto\det\(\<\operatorname{grad}^Sf_{2i+1}, \operatorname{grad}^Sf_{2j+1}\>(A)\)_{p\times p}.$$
Then $A$ is a regular point of $\Phi_{n, \mathbb{F}}$
if and only if $G_{n, \mathbb{F}}(A)\neq0$, or equivalently,
$$\operatorname{grad}^Sf_{3}\Big|_A, \cdots, \operatorname{grad}^Sf_{2p+1}\Big|_A$$
are linearly independent.
Moreover, we know that
$\mathcal{C}_{n, \mathbb{F}}=G_{n, \mathbb{F}}^{-1}(0)\cap\Phi_{n, \mathbb{F}}^{-1}(0)$
is closed in $S(n, \mathbb{F})$.
By the regular level set theorem, $\mathcal{B}_{n, \mathbb{F}}$ is a $(N(n, \mathbb{F})-p-1)$-dimensional
properly embedded submanifold of $S(n, \mathbb{F})\setminus\mathcal{C}_{n, \mathbb{F}}$.

\section{Basic properties of $\mathcal{B}_{n, \mathbb{F}}$, $n\geq 3$} \label{secproperties}
To study the submanifold geometry of $\mathcal{B}_{n, \mathbb{F}}$ in $S(n, \mathbb{F})$,
we are going to characterize the set $\mathcal{B}_{n, \mathbb{F}}$ more accurately.
Let $U(n, \mathbb{F})$ denote the Lie group
$\operatorname{SO}(n)$, $\operatorname{U}(n)$, $\operatorname{Sp}(n)$
for $\mathbb{F}=\mathbb{R}$, $\mathbb{C}$, $\mathbb{H}$, respectively.
It is well known that $$T_{I_n}U(n, \mathbb{F})=\{X\in M(n, \mathbb{F}): X^*=-X\}.$$
The isometric action of $U(n, \mathbb{F})$ on $E(n, \mathbb{F})$ given by $P\cdot A:=P^*AP$
will play an important role. The orbits are exactly the standard embeddings of the flag manifolds (cf. \cite{BCO}).
In addition, for each $A\in\Phi_{n, \mathbb{F}}^{-1}(0)$, there exists $P\in U(n, \mathbb{F})$ such that
$$P^*AP=\left\{
\begin{aligned}
\operatorname{diag}(\lambda_1, -\lambda_1, \cdots, \lambda_p, -\lambda_p, \lambda_{p+1})
\ \ \
&\mbox{if}\ n=2p+1\ \mbox{is odd};\\
\operatorname{diag}(\lambda_1, -\lambda_1, \cdots, \lambda_{p+1}, -\lambda_{p+1})
\ \ \ \ \,
&\mbox{if}\ n=2p+2\ \mbox{is even},
\end{aligned}\right.$$
where $\lambda_1\geq\cdots\geq\lambda_{p+1}\geq0$.

\begin{prop}\label{Bn}
$A\in\mathcal{B}_{n, \mathbb{F}}$ if and only if $\lambda_1, \cdots, \lambda_{p+1}$ are distinct.
\end{prop}
\begin{proof}
Note that $I_n$ is orthogonal to $$\operatorname{grad}^Sf_{2\alpha+1}\Big|_A=(2\alpha+1)\[A^{2\alpha}-\frac{1}{n}\(\operatorname{tr}A^{2\alpha}\)I_n\],\
1\leq \alpha\leq p.$$
Hence
$\operatorname{grad}^Sf_{3}\Big|_A, \cdots, \operatorname{grad}^Sf_{2p+1}\Big|_A$ are linearly independent if and only if $I_n, A^2, \cdots$, $A^{2p}$ are linearly independent.
Since $$\operatorname{grad}^Sf_{2\alpha+1}\Big|_{P^*AP}=P^*\(\operatorname{grad}^Sf_{2\alpha+1}\Big|_A\)P$$
for any $P\in U(n, \mathbb{F})$ and each $1\leq \alpha\leq p$,
we can assume without loss of generality
\begin{equation}\label{diagA}
A=\left\{
\begin{aligned}
\operatorname{diag}(\lambda_1, -\lambda_1, \cdots, \lambda_p, -\lambda_p, \lambda_{p+1})
\ \ \ &\mbox{if}\ n=2p+1\ \mbox{is odd};\\
\operatorname{diag}(\lambda_1, -\lambda_1, \cdots, \lambda_{p+1}, -\lambda_{p+1})
\ \ \ \ \, &\mbox{if}\ n=2p+2\ \mbox{is even}.
\end{aligned}\right.
\end{equation}
Therefore, $I_n, A^2, \cdots, A^{2p}$ are linearly independent if and only if
$\eta_0, \eta_1, \cdots, \eta_p$ are linearly independent, where
$\eta_\alpha=(\lambda_1^{2\alpha}, \cdots, \lambda_{p+1}^{2\alpha})\in\mathbb{R}^{p+1}$, $\lambda^0:=1$,
$0\leq \alpha\leq p$.
Then the result follows from the Vandermonde determinant
$$\det(\eta_0^t, \cdots, \eta_p^t)=\prod_{1\leq i<j\leq p+1}(\lambda_j^2-\lambda_i^2).$$
\end{proof}

By Proposition \ref{Bn}, we can characterize
$\mathcal{B}_{n, \mathbb{F}}$ and $\mathcal{C}_{n, \mathbb{F}}$ more easily.
\begin{exmp}\label{Cnexamp}
$$\begin{aligned}
&\mathcal{C}_{3, \mathbb{F}}=\varnothing,\\
&\mathcal{C}_{4, \mathbb{F}}=
U(n, \mathbb{F})\cdot\frac12\operatorname{diag}\(1, -1, 1, -1\),\\
&\mathcal{C}_{5, \mathbb{F}}=U(n, \mathbb{F})\cdot
\left\{\frac1{\sqrt2}\operatorname{diag}\(1, -1, 0, 0, 0\),
\frac12\operatorname{diag}\(1, -1, 1, -1, 0\)\right\}\\
&\mathcal{C}_{6, \mathbb{F}}=U(n, \mathbb{F})\cdot
\left\{\frac12\operatorname{diag}\(s, -s, s, -s, \sqrt{2(1-s^2)}, -\sqrt{2(1-s^2)}\)
: 0\leq s\leq 1\right\}.
\end{aligned}$$
\end{exmp}

Next, we are going to find out principal curvatures and principal directions of
the submanifold $\mathcal{B}_{n, \mathbb{F}}$ in $S(n, \mathbb{F})$.
For each $A\in\mathcal{B}_{n, \mathbb{F}}$ and each $1\leq \alpha\leq p$,
let $$X_\alpha\big|_A:=\frac1{2\alpha+2}\operatorname{grad}^Sf_{2\alpha+2}\Big|_A
=A^{2\alpha+1}-\(\operatorname{tr}A^{2\alpha+2}\)A\in T_A\mathcal{B}_{n, \mathbb{F}}.$$
To simplify notations, we denote
$$\hat{E}:=\frac1{\sqrt{2}}\(E_{(n-1)(n-1)}-E_{nn}\).$$
\begin{lem}\label{Bn TanSp}
Let $A=\operatorname{diag}(\mu_1, \cdots, \mu_n)\in\mathcal{B}_{n, \mathbb{R}}$ be as $(\ref{diagA})$,
where $|\mu_1|\geq\cdots\geq|\mu_n|\geq 0$.
\begin{itemize}
\item[(1)] If $n=2p+1$ is odd, then $T_A\mathcal{B}_{n, \mathbb{R}}$ is spanned by
$$\big\{\hat{E}_{ij}: 1\leq i<j\leq n\big\}
\cup\big\{X_\alpha\big|_A: 1\leq \alpha\leq p-1\big\}.$$
\item[(2)] If $n=2p+2$ is even and $|\mu_{n-1}|=|\mu_n|>0$,
then $T_A\mathcal{B}_{n, \mathbb{R}}$ is spanned by
$$\big\{\hat{E}_{ij}: 1\leq i<j\leq n\big\}
\cup\big\{X_\alpha\big|_A: 1\leq \alpha\leq p\big\}.$$
\item[(3)] If $n=2p+2$ is even and $|\mu_{n-1}|=|\mu_n|=0$,
then $T_A\mathcal{B}_{n, \mathbb{R}}$ is spanned by
$$\big\{\hat{E}_{ij}: 1\leq i<j\leq n\big\}
\cup\big\{X_\alpha\big|_A: 1\leq \alpha\leq p-1\big\}\cup\big\{\hat{E}\big\}.$$
\end{itemize}
\end{lem}
\begin{proof}
Since $A$ is diagonal, $V:=\operatorname{Span}\big\{\hat{E}_{ij}: 1\leq i<j\leq n\big\}$ is orthogonal to $A$ and the normal space $\operatorname{Span}\big\{\operatorname{grad}^Sf_{3}\big|_A, \cdots, \operatorname{grad}^Sf_{2p+1}\big|_A \big\}$, and thus is a $\frac12n(n-1)$-dimensional subspace of $T_A\mathcal{B}_{n, \mathbb{R}}$.
Hence we have the orthogonal decomposition $T_A\mathcal{B}_{n, \mathbb{R}}=V\oplus V^{\bot}$.
Direct calculations show that $\{X_\alpha\big|_A: 1\leq \alpha\leq p\big\}\subset V^{\bot}$ and
$\operatorname{dim}V^{\bot}=n-p-2$.

In case (1), $\operatorname{dim}V^{\bot}=p-1$.
By Proposition \ref{Bn}, $A$ has $n$ distinct eigenvalues, and thus
the minimal polynomial of $A$ has degree $n=2p+1$ and its even order terms have vanishing coefficients.
It follows that $X_1, \cdots, X_{p-1}$ are linearly independent and $X_p$ is a linear combination of them,
and hence $V^{\bot}=\operatorname{Span}\{X_\alpha\big|_A: 1\leq \alpha\leq p-1\big\}$.

In case (2), $\operatorname{dim}V^{\bot}=p$.
By Proposition \ref{Bn}, $A$ has $n$ distinct eigenvalues,
and thus the minimal polynomial of $A$ has degree $n=2p+2$.
It follows that $X_1, \cdots, X_p$ are linearly independent,
and hence $V^{\bot}=\operatorname{Span}\{X_\alpha\big|_A: 1\leq \alpha\leq p\big\}$.

In case (3), $\operatorname{dim}V^{\bot}=p$.
Since $|\mu_{n-1}|=|\mu_n|=0$,
we have $\hat{E}\in V^{\bot}$ and $\hat{E}\bot X_\alpha\big|_A$ for each $1\leq \alpha\leq p$.
By Proposition \ref{Bn}, $A$ has $n-1$ distinct eigenvalues, and thus
the minimal polynomial of $A$ has degree $n-1=2p+1$ and its even order terms have vanishing coefficients.
It follows that $X_1, \cdots, X_{p-1}$ are linearly independent and $X_p$ is a linear combination of them, and hence
$V^{\bot}=\operatorname{Span}\{X_\alpha\big|_A: 1\leq \alpha\leq p-1\big\} \cup\big\{\hat{E}\big\}$.
\end{proof}

The next lemma follows from Lemma \ref{Bn TanSp} and Section \ref{sec-pre} immediately.
\begin{lem}\label{Bn TanSp2}
Let $A=\operatorname{diag}(\mu_1, \cdots, \mu_n)\in\mathcal{B}_{n, \mathbb{R}}$ be as $(\ref{diagA})$,
and let $\mathcal{E}$ be the basis of $T_A\mathcal{B}_{n, \mathbb{R}}$ described by the above lemma.
Then $T_A\mathcal{B}_{n, \mathbb{C}}$ is spanned by $\mathcal{E}\cup\mathcal{E}_2$,
and $T_A\mathcal{B}_{n, \mathbb{H}}$ is spanned by $\mathcal{E}\cup\mathcal{E}_2\cup\mathcal{E}_3\cup\mathcal{E}_4$.
\end{lem}

\begin{thm}\label{Bn austere}
$\mathcal{B}_{n, \mathbb{F}}$ is an austere submanifold of $S(n, \mathbb{F})$ with flat normal bundle.
\end{thm}
\begin{proof}
For each $A\in\mathcal{B}_{n, \mathbb{F}}$, let $\rho_\alpha(A):=\left|\operatorname{grad}^Sf_{2\alpha+1}\right|(A)>0$ and let $$\xi_\alpha\big|_A:=\frac{1}{\rho_\alpha(A)}\operatorname{grad}^Sf_{2\alpha+1}\Big|_A
=\frac{2\alpha+1}{\rho_\alpha(A)}\[A^{2\alpha}-\frac{1}{n}\(\operatorname{tr}A^{2\alpha}\)I_n\],\
1\leq \alpha\leq p.$$
Then $(\xi_1, \cdots, \xi_p)$ is a global unit frame for the normal bundle
of $\mathcal{B}_{n, \mathbb{F}}$ in $S(n, \mathbb{F})$,
which implies that $\mathcal{B}_{n, \mathbb{F}}$ has trivial normal bundle.
In addition, the shape operator $S_{\xi_\alpha}$ in direction $\xi_\alpha$ can be computed as $$\begin{aligned}
\<S_{\xi_\alpha}(X), Y\>(A)
&=\operatorname{II}_{\xi_\alpha}(X, Y)(A)
=\frac{-1}{\rho_\alpha(A)}(\operatorname{Hess}^Sf_{2\alpha+1})_A(X,Y)\\
&=\frac{-1}{\rho_\alpha(A)}\(D^2F_{2\alpha+1}\)_A(X,Y)
=-\frac{2\alpha+1}{\rho_\alpha(A)}\sum_{r=0}^{2\alpha-1}\<XA^r, A^{2\alpha-r-1}Y\>
\end{aligned}$$
for any $X, Y\in T_A\mathcal{B}_{n, \mathbb{F}}$.
Note that for any $P\in U(n, \mathbb{F})$ and each $1\leq \alpha\leq p$, we have
$$\<S_{\xi_\alpha}(P^*XP), P^*YP\>(P^*AP)=\<S_{\xi_\alpha}(X), Y\>(A).$$
Thus we can assume $A=\operatorname{diag}(\mu_1, \cdots, \mu_n)\in\mathcal{B}_{n, \mathbb{R}}$ as (\ref{diagA}).
For any $1\leq i<j\leq n$, $1\leq k<l\leq n$ and $1\leq \alpha, \beta, \gamma\leq p$, we have
$$\<S_{\xi_\alpha}(\hat{E}_{ij}),\hat{E}_{kl}\>(A)=-\delta^{ij}_{kl}
\frac{2\alpha+1}{\rho_\alpha(A)}\sum_{r=0}^{2\alpha-1}\mu_i^{2\alpha-1-r}\mu_j^{r},$$
$$\<S_{\xi_\alpha}(\hat{E}_{ij}), X_\gamma\>(A)=0\ \mbox{and}\
\<S_{\xi_\alpha}(X_\beta), X_\gamma\>(A)=0,$$
where $\delta^{ij}_{kl}$ is the Kronecker delta for multi-indices.
For $\mathbb{F}=\mathbb{C}$, we have
$$\<S_{\xi_\alpha}(\mathbf{i}\check{E}_{ij}),\mathbf{i}\check{E}_{kl}\>(A)=-\delta^{ij}_{kl}
\frac{2\alpha+1}{\rho_\alpha(A)}\sum_{r=0}^{2\alpha-1}\mu_i^{2\alpha-1-r}\mu_j^{r},$$
$$\<S_{\xi_\alpha}(\mathbf{i}\check{E}_{ij}),\hat{E}_{kl}\>(A)=0,\
\<S_{\xi_\alpha}(\mathbf{i}\check{E}_{ij}), X_\gamma\>(A)=0.$$
For $\mathbb{F}=\mathbb{H}$, we also have
$$\<S_{\xi_\alpha}(\mathbf{q}\check{E}_{ij}),\mathbf{q}\check{E}_{kl}\>(A)=-\delta^{ij}_{kl}
\frac{2\alpha+1}{\rho_\alpha(A)}\sum_{r=0}^{2\alpha-1}\mu_i^{2\alpha-1-r}\mu_j^{r},$$
$$\<S_{\xi_\alpha}(\mathbf{q}\check{E}_{ij}),\hat{E}_{kl}\>(A)=0,\
\<S_{\xi_\alpha}(\mathbf{q}\check{E}_{ij}), X_\gamma\>(A)=0$$
for $\mathbf{q}\in\{\mathbf{i}, \mathbf{j}, \mathbf{k}\}$, and
$$\<S_{\xi_\alpha}(\mathbf{j}\check{E}_{ij}),\mathbf{i}\check{E}_{kl}\>(A)=0,\
\<S_{\xi_\alpha}(\mathbf{k}\check{E}_{ij}), \mathbf{i}\check{E}_{kl}\>(A)=0,\ \<S_{\xi_\alpha}(\mathbf{j}\check{E}_{ij}),\mathbf{k}\check{E}_{kl}\>(A)=0.$$
In addition, if $n=2p+2$ is even and $|\mu_{n-1}|=|\mu_n|=0$, we have
$$\<S_{\xi_\alpha}(\hat{E}_{ij}), \hat{E}\>(A)=0,\
\<S_{\xi_\alpha}(X_\beta), \hat{E}\>(A)=0,\
\<S_{\xi_\alpha}(\hat{E}), \hat{E}\>(A)=0$$
and $$\<S_{\xi_\alpha}(\mathbf{q}\check{E}_{ij}), \hat{E}\>(A)=0\
\mbox{for}\ \mathbf{q}\in\{\mathbf{i}, \mathbf{j}, \mathbf{k}\}.$$

In conclusion, the basis of $T_A\mathcal{B}_{n, \mathbb{F}}$ given in Lemma \ref{Bn TanSp} and Lemma \ref{Bn TanSp2} provides common principal directions
of all shape operators $S_{\xi_1}\big|_A, \cdots, S_{\xi_p}\big|_A$,
and every nonzero principal curvature of $S_{\xi_\alpha}$ at $A$ has the form
$$-\frac{2\alpha+1}{\rho_\alpha(A)}\sum_{r=0}^{2\alpha-1}\mu_i^{2\alpha-1-r}\mu_j^{r}.$$
These principal curvatures are invariant under multiplication by $-1$ as $\{\mu_1, \cdots, \mu_n\}$. Hence
$\mathcal{B}_{n, \mathbb{F}}$ is an austere submanifold with flat normal bundle.
\end{proof}

\begin{rem}
In the case of $n=3$,
$\mathcal{B}_{3, \mathbb{R}}$, $\mathcal{B}_{3, \mathbb{C}}$ and $\mathcal{B}_{3, \mathbb{H}}$
are exactly Cartan's minimal isoparametric hypersurfaces 
in $S^{4}$, $S^{7}$ and $S^{13}$, respectively.
\end{rem}

\begin{rem}
It is not hard to see that
the cone $C(\mathcal{B}_{n, \mathbb{F}})$ is an austere submanifold of $E(n, \mathbb{F})$.
Then by \cite[p.102, Theorem 3.11]{HL}, the normal bundle $NC(\mathcal{B}_{n, \mathbb{F}})$
is special Lagrangian in $TE(n, \mathbb{F})\cong\mathbb{C}^{N(n, \mathbb{F})}$. Recall also that $C(\mathcal{B}_{3, \mathbb{H}})$ is minimizing in $E(n, \mathbb{F})$ (cf. \cite{TZ20}), it is interesting to know whether this is true for $n>3$, although $\mathcal{B}_{n, \mathbb{F}}$ will not be compact.
\end{rem}

Next we verify more basic properties of $\mathcal{B}_{n, \mathbb{F}}$.
\begin{prop}\label{Bn connectness}
$\mathcal{B}_{n, \mathbb{F}}$ is orientable and connected.
\end{prop}
\begin{proof}
$\mathcal{B}_{n, \mathbb{F}}$ is orientable
because the normal bundle of $\mathcal{B}_{n, \mathbb{F}}$ is trivial.
For the connectivity, we consider the $U(n, \mathbb{F})$-action on $\mathcal{B}_{n, \mathbb{F}}$.
Since $U(n, \mathbb{F})$ is connected,
every orbit of this $U(n, \mathbb{F})$-action is path-connected. Let $\mathcal{D}_n$ denote the subset of $\mathcal{B}_{n, \mathbb{F}}$ which consists of matrices of the diagonal form (\ref{diagA}).
Then every orbit of the $U(n, \mathbb{F})$-action has a unique representative in $\mathcal{D}_n$.
For any $A, B\in\mathcal{D}_n$, let $$\gamma(t)=\frac{(1-t)A+tB}{|(1-t)A+tB|},\ t\in\[0, 1\].$$
Note that for any $1\leq i<j\leq p+1$,
$$(1-t)\lambda_i(A)+t\lambda_i(B)>(1-t)\lambda_j(A)+t\lambda_j(B).$$
Then $\gamma$ is a path in $\mathcal{D}_n$ from $A$ to $B$, and thus $\mathcal{D}_n$ is path-connected.
It follows that $\mathcal{B}_{n, \mathbb{F}}$ is connected.
\end{proof}

\begin{prop}\label{Bn substantial}
The inclusion map $\iota: \mathcal{B}_{n, \mathbb{F}}\rightarrow E(n, \mathbb{F})$ is substantial.
\end{prop}
\begin{proof}
Fix any $A\in\mathcal{D}_n$, then
Lemma \ref{Bn TanSp} implies that
$\mathcal{E}_1\subset T_A\mathcal{B}(n, \mathbb{F})$.
For each $2\leq i\leq n$, there exists $P_i\in\operatorname{SO}(n)$ such that
$P_i^*\hat{E}_{1i}P_i=\hat{E}_{i}$.
Note that $\hat{E}_{i}\in T_{P_i^*AP_i}\mathcal{B}(n, \mathbb{F})$ and
$\mathcal{E}_0\cup\mathcal{E}_1$ is a basis of $E(n, \mathbb{R})$.
For $\mathbb{F}=\mathbb{C}$, we have $\mathcal{E}_2\subset T_A\mathcal{B}(n, \mathbb{C})$ and
$\mathcal{E}_0\cup\mathcal{E}_1\cup\mathcal{E}_2$ is a basis of $E(n, \mathbb{C})$.
For $\mathbb{F}=\mathbb{H}$, we also have
$\mathcal{E}_3, \mathcal{E}_4\subset T_A\mathcal{B}(n, \mathbb{H})$ and
$\mathcal{E}_0\cup\mathcal{E}_1\cup\mathcal{E}_2\cup\mathcal{E}_3\cup\mathcal{E}_4$
is a basis of $E(n, \mathbb{H})$. Hence the tangent bundle of the orbit through $A\in\mathcal{D}_n$  of the $U(n, \mathbb{F})$-action contains a basis of $E(n, \mathbb{F})$.
It follows that these orbits and thus $\mathcal{B}_{n, \mathbb{F}}$  are not contained in any hyperplane of $E(n, \mathbb{F})$.
\end{proof}

\begin{cor}\label{Bnfull}
$\mathcal{B}_{n, \mathbb{F}}$ is a full submanifold of $S(n, \mathbb{F})$.
\end{cor}

\section{New examples of irreducible proper Dupin hypersyrfaces} \label{sec-Dupin}
\subsection{$\mathcal{B}_{4, \mathbb{F}}$ and $\mathcal{C}_{4, \mathbb{F}}$}\
Note that $E(4, \mathbb{F})\cong \mathbb{R}^{9}, \mathbb{R}^{15}, \mathbb{R}^{27}$ and thus $S(4, \mathbb{F})= S^{8}, S^{14}, S^{26}$ for $\mathbb{F}=\mathbb{R}, \mathbb{C}, \mathbb{H}$ respectively. By Theorem \ref{Bn austere}, Proposition \ref{Bn connectness} and Corollary \ref{Bnfull}, $\mathcal{B}_{4, \mathbb{F}}\subset S(4, \mathbb{F})$ is an orientable, connected, non-totally-geodesic, austere hypersurface.
Since $U(n, \mathbb{F})$ is compact,
as an orbit of this compact Lie group action, the critical point set $$\mathcal{C}_{4, \mathbb{F}}=U(4, \mathbb{F})\cdot\frac12\operatorname{diag}\(1, -1, 1, -1\)$$
in Example \ref{Cnexamp} is a properly embedded submanifold of $S(4, \mathbb{F})$.

\begin{rem}\label{C4Grassmann}
It is not hard to compute that for $A=\frac12\operatorname{diag}\(1, 1, -1, -1\)\in\mathcal{C}_{4, \mathbb{F}}$, the isotropy group
$$G_A=\left\{\begin{aligned}
\big\{\operatorname{diag}(P_1, P_2)\in\operatorname{SO}(4): P_1, P_2\in\operatorname{O(2)}\big\}
\ \ \ \ \ \, &\mbox{if}\ \mathbb{F}=\mathbb{R};\\
\big\{\operatorname{diag}(P_1, P_2)\in U(4, \mathbb{F}): P_1, P_2\in U(2, \mathbb{F})\big\}
\ \ \ &\mbox{if}\ \mathbb{F}=\mathbb{C}\ \mbox{or}\ \mathbb{H},
\end{aligned}\right.$$
and thus $\mathcal{C}_{4, \mathbb{F}}$ is diffeomorphic to the Grassmann manifold $\operatorname{G}_2(\mathbb{F}^4)$.
\end{rem}

\begin{prop}\label{Cn substantial}
The inclusion map $\iota: \mathcal{C}_{4, \mathbb{F}}\rightarrow E(4, \mathbb{F})$ is substantial.
\end{prop}
\begin{proof}
Suppose that there exists a matrix $B=(B_{ij})_{4\times4}\in E(4, \mathbb{F})$ such that
for each $A\in\mathcal{C}_{4, \mathbb{F}}$, $$\<A, B\>=0.$$
Since $A$ can take values $\frac12\operatorname{diag}(1, -1, 1, -1)$,
$\frac12\operatorname{diag}(1, 1, -1, -1)$, $\frac12\operatorname{diag}(1, -1, -1, 1)$
and $\operatorname{tr}B=0$, we have $B_{11}=B_{22}=B_{33}=B_{44}=0$.
Continue to take $A$ as
$$\frac12\begin{pmatrix}
0 & 1 & 0 & 0\\
1 & 0 & 0 & 0\\
0 & 0 & 1 & 0\\
0 & 0 & 0 & -1\\
\end{pmatrix},
\frac12\begin{pmatrix}
0 & 0 & 1 & 0\\
0 & 1 & 0 & 0\\
1 & 0 & 0 & 0\\
0 & 0 & 0 & -1\\
\end{pmatrix},
\frac12\begin{pmatrix}
0 & 0 & 0 & 1\\
0 & 1 & 0 & 0\\
0 & 0 & -1 & 0\\
1 & 0 & 0 & 0\\
\end{pmatrix},$$
$$\frac12\begin{pmatrix}
1 & 0 & 0 & 0\\
0 & 0 & 1 & 0\\
0 & 1 & 0 & 0\\
0 & 0 & 0 & -1\\
\end{pmatrix},
\frac12\begin{pmatrix}
1 & 0 & 0 & 0\\
0 & 0 & 0 & 1\\
0 & 0 & -1 & 0\\
0 & 1 & 0 & 0\\
\end{pmatrix},
\frac12\begin{pmatrix}
1 & 0 & 0 & 0\\
0 & -1 & 0 & 0\\
0 & 0 & 0 & 1\\
0 & 0 & 1 & 0\\
\end{pmatrix},$$
we have $\mathfrak{R}(B)=0$. 
For $\mathbb{F}=\mathbb{R}$, this already means that $B=0$.
For $\mathbb{F}=\mathbb{C}$, $A$ can take values
$$\frac12\begin{pmatrix}
0 & \mathbf{i} & 0 & 0\\
-\mathbf{i} & 0 & 0 & 0\\
0 & 0 & 1 & 0\\
0 & 0 & 0 & -1\\
\end{pmatrix},
\frac12\begin{pmatrix}
0 & 0 & \mathbf{i} & 0\\
0 & 1 & 0 & 0\\
-\mathbf{i} & 0 & 0 & 0\\
0 & 0 & 0 & -1\\
\end{pmatrix},
\frac12\begin{pmatrix}
0 & 0 & 0 & \mathbf{i}\\
0 & 1 & 0 & 0\\
0 & 0 & -1 & 0\\
-\mathbf{i} & 0 & 0 & 0\\
\end{pmatrix},$$
$$\frac12\begin{pmatrix}
1 & 0 & 0 & 0\\
0 & 0 & \mathbf{i} & 0\\
0 & -\mathbf{i} & 0 & 0\\
0 & 0 & 0 & -1\\
\end{pmatrix},
\frac12\begin{pmatrix}
1 & 0 & 0 & 0\\
0 & 0 & 0 & \mathbf{i}\\
0 & 0 & -1 & 0\\
0 & -\mathbf{i} & 0 & 0\\
\end{pmatrix},
\frac12\begin{pmatrix}
1 & 0 & 0 & 0\\
0 & -1 & 0 & 0\\
0 & 0 & 0 & \mathbf{i}\\
0 & 0 & -\mathbf{i} & 0\\
\end{pmatrix},$$
and thus we have $B=0$.
For $\mathbb{F}=\mathbb{H}$, $A$ can take values
$$\frac12\begin{pmatrix}
0 & \mathbf{q} & 0 & 0\\
-\mathbf{q} & 0 & 0 & 0\\
0 & 0 & 1 & 0\\
0 & 0 & 0 & -1\\
\end{pmatrix},
\frac12\begin{pmatrix}
0 & 0 & \mathbf{q} & 0\\
0 & 1 & 0 & 0\\
-\mathbf{q} & 0 & 0 & 0\\
0 & 0 & 0 & -1\\
\end{pmatrix},
\frac12\begin{pmatrix}
0 & 0 & 0 & \mathbf{q}\\
0 & 1 & 0 & 0\\
0 & 0 & -1 & 0\\
-\mathbf{q} & 0 & 0 & 0\\
\end{pmatrix},$$
$$\frac12\begin{pmatrix}
1 & 0 & 0 & 0\\
0 & 0 & \mathbf{q} & 0\\
0 & -\mathbf{q} & 0 & 0\\
0 & 0 & 0 & -1\\
\end{pmatrix},
\frac12\begin{pmatrix}
1 & 0 & 0 & 0\\
0 & 0 & 0 & \mathbf{q}\\
0 & 0 & -1 & 0\\
0 & -\mathbf{q} & 0 & 0\\
\end{pmatrix},
\frac12\begin{pmatrix}
1 & 0 & 0 & 0\\
0 & -1 & 0 & 0\\
0 & 0 & 0 & \mathbf{q}\\
0 & 0 & -\mathbf{q} & 0\\
\end{pmatrix}$$
for $\mathbf{q}\in\{\mathbf{i}, \mathbf{j}, \mathbf{k}\}$, and thus we also have $B=0$.
It follows that $\mathcal{C}_{4, \mathbb{F}}$
is not contained in any hyperplane of $E(4, \mathbb{F})$.
\end{proof}

\begin{cor}
$\mathcal{C}_{4, \mathbb{F}}$ is a full submanifold of $S(4, \mathbb{F})$.
\end{cor}

\begin{prop}\label{B4 scalar}
The scalar curvature of $\mathcal{B}_{4, \mathbb{F}}$ is
$$s(A)=6m(6m+1)-2m\(\operatorname{tr}A^4-\frac14\)^{-1}.$$
\end{prop}
\begin{proof}
For each $A\in \mathcal{B}_{4, \mathbb{F}}$, we have
$$\operatorname{grad}^Sf_3(A)=3A^2-\frac34I_4,$$
and thus $$\rho_1(A)=\left|\operatorname{grad}^Sf_3(A)\right|=3\(\operatorname{tr}A^4-\frac14\)^{\frac12}.$$
Assume that $A=P^*\operatorname{diag}(\lambda_1, -\lambda_1, \lambda_2, -\lambda_2)P$,
where $P\in U(4, \mathbb{F})$ and $\lambda_1>\lambda_2\geq0$.
By the proof of Theorem \ref{Bn austere}, the distinct principal curvatures of $S_{\xi_1}$ at $A$ are
$$\frac{3}{\rho_1(A)}\(\lambda_1+\lambda_2\), \frac{-3}{\rho_1(A)}\(\lambda_1+\lambda_2\),
\frac{3}{\rho_1(A)}\(\lambda_1-\lambda_2\), \frac{-3}{\rho_1(A)}\(\lambda_1-\lambda_2\), 0,$$
with respective multiplicities
$m$, $m$, $m$, $m$ and $2m+1$.
Then we have $$\left|\operatorname{II}\right|^2(A)
=\frac{9}{\rho_1^2(A)} m \[2\(\lambda_1+\lambda_2\)^2+2\(\lambda_1-\lambda_2\)^2\]
=2m\(\operatorname{tr}A^4-\frac14\)^{-1}.$$
Therefore, the Gauss equation of the austere hypersurface $\mathcal{B}_{4, \mathbb{F}}$ in $S(4, \mathbb{F})$ implies that
the scalar curvature of $\mathcal{B}_{4, \mathbb{F}}$ is
$$\begin{aligned}
s(A)&=(N(4,\mathbb{F})-2)(N(4,\mathbb{F})-3)-\left|\operatorname{II}\right|^2(A)\\
&=6m(6m+1)-2m\(\operatorname{tr}A^4-\frac14\)^{-1}.
\end{aligned}$$
\end{proof}

\begin{rem}
For $A\in\mathcal{A}_4$, the Cauchy-Schwarz inequality implies that
$$\operatorname{tr}A^4\geq\frac14\(\operatorname{tr}A^2\)^2=\frac14,$$ and $\operatorname{tr}A^4=\frac14$ if and only if $A\in\mathcal{C}_{4, \mathbb{R}}$.
By Proposition $\ref{B4 scalar}$, we have
$$\lim_{\operatorname{tr}A^4\rightarrow\frac14}s(A)=-\infty.$$
Unlike Cartan's isoparametric hypersurface in the case of $n=3$ where $\mathcal{A}_3=\mathcal{B}_{3, \mathbb{R}}$, now $\mathcal{A}_4=\mathcal{B}_{4, \mathbb{R}}\cup \mathcal{C}_{4, \mathbb{R}}$ is not a smooth embedded submanifold of $S^8$.
\end{rem}

\subsection{New examples of proper Dupin hypersurfaces}
For $n=2p+2$, let $$\widetilde{\mathcal{D}}_n:=
\{\operatorname{diag}(\lambda_1, -\lambda_1, \cdots, \lambda_{p+1}, -\lambda_{p+1})\in\mathcal{D}_n:
\lambda_1>\cdots>\lambda_p>\lambda_{p+1}>0\},$$
and let $\widetilde{\mathcal{B}}_{n, \mathbb{F}}:=U(4, \mathbb{F})\cdot\widetilde{\mathcal{D}}_n$
be the union of principal orbits.
Then
$\widetilde{\mathcal{B}}_{n, \mathbb{F}}=\mathcal{B}_{n, \mathbb{F}}\cap\operatorname{GL}(n, \mathbb{F})$
is an open subset of $\mathcal{B}_{n, \mathbb{F}}$,
and thus $\widetilde{\mathcal{B}}_{n, \mathbb{F}}$ is still
a $(N(n, \mathbb{F})-p-1)$-dimensional embedded submanifold of $S(n, \mathbb{F})$.
From the proof of Proposition \ref{Bn connectness}, we know that
$\widetilde{\mathcal{B}}_{n, \mathbb{F}}$ is orientable and connected.

In the proof of Proposition \ref{B4 scalar},
we observe that on the hypersurface $\widetilde{\mathcal{B}}_{4, \mathbb{F}}\subset S(4, \mathbb{F})$, there are $5$ distinct principal curvature functions
\begin{equation}\label{princurv}
\begin{aligned}
\kappa_1(A)=\frac{3}{\rho_1(A)}\(\lambda_1+\lambda_2\), &\quad \kappa_2(A)=-\frac{3}{\rho_1(A)}\(\lambda_1+\lambda_2\),\\
\kappa_3(A)=\frac{3}{\rho_1(A)}\(\lambda_1-\lambda_2\), &\quad
\kappa_4(A)=-\frac{3}{\rho_1(A)}\(\lambda_1-\lambda_2\), &\quad \kappa_5(A)=0
\end{aligned}
\end{equation}
 with respective multiplicities
$m$, $m$, $m$, $m$ and $2m+1$.
We denote the corresponding principal distributions by
$T_{\kappa_1}$, $T_{\kappa_2}$, $T_{\kappa_3}$, $T_{\kappa_4}$ and $T_{\kappa_5}$.
By the following theorem, we obtain that
$\widetilde{\mathcal{B}}_{4, \mathbb{C}}$ and $\widetilde{\mathcal{B}}_{4, \mathbb{H}}$ are
proper Dupin hypersurface of $S(4, \mathbb{C})$ and $S(4, \mathbb{H})$, since $m=2,4$, respectively.

\begin{thm}\cite[Theorem 2.10]{CR15}\label{Dupin criterion}
Let $f: M^n\rightarrow \widetilde{M}^{n+1}$ be an oriented hypersurface of a real space form.
Suppose that $\lambda$ is a smooth principal curvature function of constant multiplicity $m>1$ on $M$.
Then the principal distribution $T_\lambda$ is integrable,
and $X\lambda=0$ for each $X\in T_\lambda$ at every point of $M$.
\end{thm}

To determine if $\widetilde{\mathcal{B}}_{4, \mathbb{R}}$ is Dupin or not,
we need to check the curvature circles of
$T_{\kappa_1}$, $T_{\kappa_2}$, $T_{\kappa_3}$ and $T_{\kappa_4}$.
Fix any $A=P^*DP\in\widetilde{\mathcal{B}}_{4, \mathbb{R}}$,
where $P\in\operatorname{SO}(4)$ and
$D=\operatorname{diag}(\mu_1, \mu_2, \mu_3, \mu_4)\in\widetilde{\mathcal{D}}_4$.
Then the proof of Theorem \ref{Bn austere} implies that $$T_{\kappa_1}(A)=\operatorname{Span}\big\{P^*\hat{E}_{13}P\big\},
T_{\kappa_2}(A)=\operatorname{Span}\big\{P^*\hat{E}_{24}P\big\},$$
$$T_{\kappa_3}(A)=\operatorname{Span}\big\{P^*\hat{E}_{14}P\big\},
T_{\kappa_4}(A)=\operatorname{Span}\big\{P^*\hat{E}_{23}P\big\}$$
$$\mbox{and}\ T_{\kappa_5}(A)=\operatorname{Span}\big\{P^*\hat{E}_{12}P, P^*\hat{E}_{34}P, X_1\big|_A\big\}.$$
For any $1\leq i<j\leq 4$, let $$\gamma_{ij}(t):=\(\exp(t\check{E}_{ij})P\)^*D(\exp(t\check{E}_{ij})P).$$ 
Then we have $\gamma_{ij}(0)=A$ and
$$\begin{aligned}
\frac{d}{dt}\gamma_{ij}(t)
&=(\check{E}_{ij}\exp(t\check{E}_{ij})P)^*D(\exp(t\check{E}_{ij})P)
+(\exp(t\check{E}_{ij})P)^*D(\check{E}_{ij}\exp(t\check{E}_{ij})P)\\
&=(\exp(t\check{E}_{ij})P)^*(D\check{E}_{ij}-\check{E}_{ij}D)(\exp(t\check{E}_{ij})P)\\
&=(\mu_i-\mu_j)(\exp(t\check{E}_{ij})P)^*\hat{E}_{ij}(\exp(t\check{E}_{ij})P).
\end{aligned}$$
Hence $\gamma_{13}$, $\gamma_{24}$, $\gamma_{14}$ and $\gamma_{23}$ are the curvature circles
corresponding to $T_{\kappa_1}$, $T_{\kappa_2}$, $T_{\kappa_3}$ and $T_{\kappa_4}$ through $A$, respectively.
It follows from $\gamma_{ij}(t)\in\operatorname{SO}(4)\cdot A$ that
all the principal curvatures are constant along these curvature circles.

To sum up the above discussion, we have the following theorem.
\begin{thm}\label{B4_pD}
$\widetilde{\mathcal{B}}_{4, \mathbb{F}}$ is a $(6m+1)$-dimensional orientable, connected, noncompact, austere, proper Dupin hypersurface of the unit sphere $S(4, \mathbb{F})$ with $5$ distinct principal curvatures of multiplicities $m$, $m$, $m$, $m$ and $2m+1$, where $m=\operatorname{dim}_{\mathbb{R}}(\mathbb{F})=1,2,4$.
\end{thm}

\subsection{Irreducibility of the proper Dupin hypersurface $\widetilde{\mathcal{B}}_{4, \mathbb{F}}$}
It is natural and convenient to study Dupin hypersurfaces by Lie sphere geometry (cf. \cite{Cecil, Cecil23}).
For $d\geq 1$, let $\langle\cdot, \cdot\rangle_L$ denote the bilinear form on $\mathbb{R}^{d+3}$, called the Lie metric, such that
$$\langle e_i, e_j\rangle_L=\left\{
\begin{aligned}
0\ \ \ &\mbox{if}\ i\neq j;\\
1\ \ \ &\mbox{if}\ 2\leq i=j\leq d+2;\\
-1\ \ \ &\mbox{if}\ i=j=1\ \mbox{or}\ d+3,
\end{aligned}\right.$$
where $(e_1,\cdots, e_{d+3})$ is the standard basis of $\mathbb{R}^{d+3}$.
Let $\Lambda^{2d-1}$ denote the manifold of projective lines on the Lie quadric $Q^{d+1}:=\{[x]\in\mathbb{R}\mathbf{P}^{d+2}: \langle x,x\rangle_L=0\}$.
For any fixed orthonormal basis of $E(4, \mathbb{F})$,
we can identity $E(4, \mathbb{F})$ with the subspace spanned by $\{e_2,\cdots, e_{N(4, \mathbb{F})+1}\}$ in $\mathbb{R}^{N(4, \mathbb{F})+2}$.
The Legendre lift of the inclusion map
$\iota: \widetilde{\mathcal{B}}_{4, \mathbb{F}}\rightarrow S(4, \mathbb{F})$ is
$$\mu: \widetilde{\mathcal{B}}_{4, \mathbb{F}}\rightarrow\Lambda^{2N(4, \mathbb{F})-3},\
A\mapsto[Z_1(A), Z_{N(4, \mathbb{F})+2}(A)],$$
where $Z_1(A)=(1, A, 0)$, $Z_{N(4, \mathbb{F})+2}(A)=(0, \xi_1(A), 1)$.
Then the 5 distinct principal curvatures $\kappa_1(A), \cdots, \kappa_5(A)$
of $\widetilde{\mathcal{B}}_{4, \mathbb{F}}$ in (\ref{princurv}) correspond to
the 5 distinct curvature spheres $$[K_i]=[\kappa_iZ_1+Z_{N(4, \mathbb{F})+2}],\ 1\leq i\leq5.$$

\begin{thm}\cite[Theorem 5.12]{Cecil}\label{reducibility}
A connected proper Dupin submanifold $\mu: W^{d-1}\rightarrow\Lambda^{2d-1}$ is reducible
if and only if there exists a curvature sphere $[K]$ of $\mu$ that lies in a linear subspace of $\mathbb{R}\mathbf{P}^{d+2}$ of codimension at least two.
\end{thm}

\begin{thm}\label{B4_irr}
$\mu: \widetilde{\mathcal{B}}_{4, \mathbb{F}}\rightarrow\Lambda^{2N(4, \mathbb{F})-3}$
is an irreducible proper Dupin submanifold.
\end{thm}
\begin{proof}
Suppose that there exist a curvature sphere $[K_i]$ of $\mu$ and
a vector $(a, B, c)\in\mathbb{R}^{N(4, \mathbb{F})+2}$, where $a,c\in\mathbb{R}$ and $B\in E(4, \mathbb{F})$, such that
for each $A\in\widetilde{\mathcal{B}}_{4, \mathbb{F}}$, $$\ \langle(a, B, c), K_i(A)\ \rangle_L=0,$$ or equivalently,
\begin{equation}\label{eq red1}
\kappa_i(A)\<A, B\>+\<\xi_1(A), B\>=a\kappa_i(A)+c.
\end{equation}

If $1\leq i\leq 4$, then $\kappa_i(A)$ is nonzero on $\widetilde{\mathcal{B}}_{4, \mathbb{F}}$.
Since $-A\in\operatorname{SO}(4)\cdot A \subset\widetilde{\mathcal{B}}_{4, \mathbb{F}}$, we have
\begin{equation}\label{eq red2}
\kappa_i(-A)\<-A, B\>+\<\xi_1(-A), B\>=a\kappa_i(-A)+c,
\end{equation}
and $\kappa_i(-A)=\kappa_i(A)$. Note that $\xi_1(-A)=\xi_1(A)$.
Thus equation (\ref{eq red1}) and (\ref{eq red2}) imply that
$\<A, B\>=0$ for each $A\in\widetilde{\mathcal{B}}_{4, \mathbb{F}}$.
In the proof of Proposition \ref{Bn substantial}, we known that
every orbit in $\widetilde{\mathcal{B}}_{4, \mathbb{F}}$
is not contained in any hyperplane of $E(4, \mathbb{F})$.
It follows that $B=0$ and equation (\ref{eq red1}) becomes
$$0=a\kappa_i(A)+c.$$
Since $\kappa_i(A)$ is nonconstant on $\widetilde{\mathcal{B}}_{4, \mathbb{F}}$, we also have $a=c=0$.
Therefore, $[K_i]$ is not contained in any proper linear subspace of
$\mathbb{R}\mathbf{P}^{N(4, \mathbb{F})+1}$.

If $i=5$, then $\kappa_5(A)\equiv0$ implies that equation (\ref{eq red1}) becomes
$$\<\xi_1(A), B\>=c.$$
Direct computations show that for
$A=P^*\operatorname{diag}(\lambda_1, -\lambda_1, \lambda_2, -\lambda_2)P
\in\widetilde{\mathcal{B}}_{4, \mathbb{F}}$,
we have $$\xi_1(A)=\frac12P^*\operatorname{diag}(1, 1, -1, -1)P,$$
i.e. the Gauss image of $\widetilde{\mathcal{B}}_{4, \mathbb{F}}$ is $\mathcal{C}_{4, \mathbb{F}}$.
Since $\mathcal{C}_{4, \mathbb{F}}$ is invariant under antipodal map, we have $c=0$.
This implies $B=0$ as $\mathcal{C}_{4, \mathbb{F}}$ is substantial by Proposition \ref{Cn substantial}.
Note that for any $a\in\mathbb{R}$, $(a, 0, 0)$ satisfies equation (\ref{eq red1}) for $\kappa_5(A)\equiv0$.
Therefore, $[K_5]$ is contained in a linear subspace of $
\mathbb{R}\mathbf{P}^{N(4, \mathbb{F})+1}$ of codimension 1,
but it is not contained in any linear subspace of codimension 2.

Combining the discussion of the above two cases,
we conclude by Theorem \ref{reducibility} that $\widetilde{\mathcal{B}}_{4, \mathbb{F}}$ is irreducible.
\end{proof}
In conclusion, we have constructed three examples of irreducible proper Dupin hypersurfaces $\widetilde{\mathcal{B}}_{4, \mathbb{F}}$ in unit spheres $S(4, \mathbb{F})$ with $g=5$ distinct principal curvatures of different multiplicities, answering negatively Thorbergsson's question which is neither true for weak irreducibility as we shall show in the following.
\subsection{Weak irreducibility} By proving the non-integrability of the conullity of each principal distribution, we verify that
the proper Dupin hypersurface $\widetilde{\mathcal{B}}_{4, \mathbb{R}}$ is not weakly reducible.
To do this, we need the following preparations.
Define
$$\varphi: \operatorname{SO}(n)\times\widetilde{\mathcal{D}}_n\rightarrow\widetilde{\mathcal{B}}_{n, \mathbb{R}},\
(P, A)\mapsto P^*AP.$$
We obverse that $\widetilde{\mathcal{D}}_n$ is an open subset of the $p$-dimensional sphere
$$\left\{\sum\limits_{i=1}^{p+1}x_i(E_{(2i-1)(2i-1)}-E_{(2i)(2i)}):\sum\limits_{i=1}^{p+1}x_i^2=\frac12\right\},$$
and for each $A\in\widetilde{\mathcal{D}}_n$,
$T_A\widetilde{\mathcal{D}}_n$ is spanned by $\big\{X_\alpha\big|_A: 1\leq \alpha\leq p\big\}$ as in (2) of Lemma \ref{Bn TanSp}.

\begin{lem}\label{Bn nbhd}
For each $A\in\widetilde{\mathcal{D}}_n$,
there exist an open neighborhood $U$ of $I_n$ in $\operatorname{SO}(n)$ and
an open neighborhood $V$ of $A$ in $\widetilde{\mathcal{D}}_n$ such that
$\varphi\big|_{U\times V}$ is a diffeomorphism onto its image.
\end{lem}
\begin{proof}
Fix any $A=\operatorname{diag}(\mu_1, \cdots, \mu_n)\in\widetilde{\mathcal{D}}_n$,
and let $\gamma(t)=(P(t), A(t))$ be a smooth curve in $\operatorname{SO}(n)\times\widetilde{\mathcal{D}}_n$ such that $\gamma(0)=(I_n, A)$ and
$\gamma'(0)=(X, Y)\in T_{I_n}\operatorname{SO}(n)\oplus T_A\widetilde{\mathcal{D}}_n$.
Then we have
$$\begin{aligned}
d\varphi_{(I_n, A)}(X, Y)&=\frac{d}{dt}\bigg|_{t=0}\varphi(\gamma(t))
=\frac{d}{dt}\bigg|_{t=0}P^*(t)A(t)P(t)\\
&=X^*A+Y+AX=AX-XA+Y.
\end{aligned}$$
By direct computations, we have
$$d\varphi_{(I_n, A)}(\check{E}_{ij},0)=(\mu_i-\mu_j)\hat{E}_{ij},$$
$$d\varphi_{(I_n, A)}(0,X_\alpha\big|_A)=X_\alpha\big|_A.$$
It follows that $d\varphi_{(I_n, A)}$ is surjective.
Since $\operatorname{dim}(\operatorname{SO}(n)\times\widetilde{\mathcal{D}}_n)
=\operatorname{dim}\widetilde{\mathcal{B}}_{n, \mathbb{R}},$
we know that $\varphi$ is a diffeomorphism on some neighborhoods of $(I_n, A)$.
\end{proof}

For each $A\in\widetilde{\mathcal{D}}_n$, by Lemma \ref{Bn nbhd}, there exists an open neighborhood $W$
of $A$ in $\widetilde{\mathcal{B}}_{n, \mathbb{R}}$ such that
$$Y_{ij}: W\rightarrow T\widetilde{\mathcal{B}}_{n, \mathbb{R}},\
B\mapsto \pi_1(\varphi^{-1}(B))\cdot \hat{E}_{ij}$$
is a well-defined smooth vector field on $W$, where $\pi_1:\operatorname{SO}(n)\times\widetilde{\mathcal{D}}_n\rightarrow\operatorname{SO}(n)$
denotes the projection onto the first factor.
It is easy to see that $$\big\{Y_{ij}: 1\leq i<j\leq n\big\}
\cup\big\{X_\alpha: 1\leq \alpha\leq p\big\}$$ forms a smooth orthogonal frame for $T\widetilde{\mathcal{B}}_{n, \mathbb{R}}$ over $W$.

\begin{prop}\label{B4R_wr}
$\widetilde{\mathcal{B}}_{4, \mathbb{R}}$ is weakly irreducible.
\end{prop}
\begin{proof}
Fix any $A=\operatorname{diag}(\mu_1, \mu_2, \mu_3, \mu_4)\in\widetilde{\mathcal{D}}_4$.
Let $\nabla$ and $\nabla^E$ denotes the Levi-Civita connection of $\widetilde{\mathcal{B}}_{4, \mathbb{R}}$ and $E(4, \mathbb{R})$, respectively.
Then for any smooth vector fields $X, Y$ on $\widetilde{\mathcal{B}}_{4, \mathbb{R}}$, we have
$$
\(\nabla_{X}Y\)(A)=\(\nabla^E_{\widetilde{X}}\widetilde{Y}\)^{\top}(A)
=\frac1a\(\left.\frac{d}{dt}\right|_{t=0}Y(\gamma(t))\)^{\top},$$
where $\widetilde{X}, \widetilde{Y}$ are arbitrary extensions of $X, Y$ in a neighborhood of $A$
in $E(4, \mathbb{R})$, respectively,
and $\gamma$ is a smooth curve on $\widetilde{\mathcal{B}}_{4, \mathbb{R}}$ such that
$\gamma(0)=A$ and $\gamma'(0)=aX(A)$ for some $a\in\mathbb{R}\setminus\{0\}$.
For any $1\leq i<j\leq 4$, let $$\gamma_{ij}(t):=\exp(-t\check{E}_{ij})A\exp(t\check{E}_{ij}).$$
Then $\gamma_{ij}(0)=A$ and for sufficiently small $t$,
$$\begin{aligned}
\frac{d}{dt}\gamma_{ij}(t)
=(\mu_i-\mu_j)\exp(-t\check{E}_{ij})\hat{E}_{ij}(\exp(t\check{E}_{ij}))
=(\mu_i-\mu_j)Y_{ij}(\gamma_{ij}(t)).
\end{aligned}$$
Hence we have $$\begin{aligned}
\(\nabla_{Y_{ij}}Y_{kl}\)(A)&=\frac1{\mu_i-\mu_j}\(\left.\frac{d}{dt}\right|_{t=0}Y_{kl}(\gamma_{ij}(t))\)^{\top}\\
&=\frac1{\mu_i-\mu_j}\(\left.\frac{d}{dt}\right|_{t=0}\exp(-t\check{E}_{ij})\hat{E}_{kl}\exp(t\check{E}_{ij})\)^{\top}\\
&=\frac1{\mu_i-\mu_j}\(\hat{E}_{kl}\check{E}_{ij}-\check{E}_{ij}\hat{E}_{kl}\)^{\top}\\
&=\frac12\frac1{\mu_i-\mu_j}\big[\delta_{ik}(E_{jl}+E_{lj})+\delta_{il}(E_{jk}+E_{kj})\\
&\quad\quad\quad\quad\quad\quad-\delta_{jk}(E_{il}+E_{li})-\delta_{jl}(E_{ik}+E_{ki})\big].
\end{aligned}$$
As in the proof of Theorem \ref{B4_pD}, one can see that $Y_{ij}(t)$ are the principal directions of $\widetilde{\mathcal{B}}_{4, \mathbb{R}}$ along $\gamma_{ij}(t)$, i.e.,
$$ T_{\kappa_1}=\operatorname{Span}\{Y_{13}\},~T_{\kappa_2}=\operatorname{Span}\{Y_{24}\},~T_{\kappa_3}=\operatorname{Span}\{Y_{14}\},~~T_{\kappa_4}=\operatorname{Span}\{Y_{23}\},$$
$$~T_{\kappa_5}=\operatorname{Span}\{Y_{12},Y_{34}, X_1\}.$$
By direct computations we obtain
$$[Y_{12}, Y_{23}](A)=\frac{-1}{\sqrt{2}}\frac{\mu_1-\mu_3}{(\mu_1-\mu_2)(\mu_2-\mu_3)}\hat{E}_{13}
\in T_{\kappa_1}\setminus\{0\},$$
$$[Y_{12}, Y_{14}](A)=\frac1{\sqrt{2}}\frac{\mu_2-\mu_4}{(\mu_1-\mu_2)(\mu_1-\mu_4)}\hat{E}_{24}
\in T_{\kappa_2}\setminus\{0\},$$
$$[Y_{12}, Y_{24}](A)=\frac{-1}{\sqrt{2}}\frac{\mu_1-\mu_4}{(\mu_1-\mu_2)(\mu_2-\mu_4)}\hat{E}_{14}
\in T_{\kappa_3}\setminus\{0\},$$
$$[Y_{12}, Y_{13}](A)=\frac1{\sqrt{2}}\frac{\mu_2-\mu_3}{(\mu_1-\mu_2)(\mu_1-\mu_3)}\hat{E}_{23}
\in T_{\kappa_4}\setminus\{0\},$$
$$[Y_{13}, Y_{14}](A)=\frac1{\sqrt{2}}\frac{\mu_3-\mu_4}{(\mu_1-\mu_3)(\mu_1-\mu_4)}\hat{E}_{34}
\in T_{\kappa_5}\setminus\{0\}.$$
It follows from the Frobenius theorem that
$T_{\kappa_1}^{\bot}$, $T_{\kappa_2}^{\bot}$, $T_{\kappa_3}^{\bot}$, $T_{\kappa_4}^{\bot}$ and $T_{\kappa_5}^{\bot}$ are non-integrable. Thus $\widetilde{\mathcal{B}}_{4, \mathbb{R}}$ is weakly irreducible.
\end{proof}

\section{Compact austere submanifolds $\mathcal{C}_{4, \mathbb{F}}$ in $S(4, \mathbb{F})$}\label{secC4}
The austere submanifolds $\mathcal{B}_{n, \mathbb{F}}$ given in previous sections are noncompact if $n\geq4$. In this section, we show that the three compact submanifolds $\mathcal{C}_{4, \mathbb{F}}$ in $S(4, \mathbb{F})$ are austere. In Remark \ref{C4Grassmann} we have mentioned that $\mathcal{C}_{4, \mathbb{F}}$ is topologically the Grassmann manifold $\operatorname{G}_2(\mathbb{F}^4)$. It worths to remark further that these austere submanifolds $\mathcal{C}_{4, \mathbb{F}}$ are not the focal submanifolds of any isoparametric hypersurface in $S(4, \mathbb{F})$.

Fix $A=\frac12\operatorname{diag}\(1, 1, -1, -1\)\in\mathcal{C}_{4, \mathbb{F}}$, and let
$$\phi: U(4, \mathbb{F})\rightarrow \mathcal{C}_{4, \mathbb{F}},\ P\mapsto P^*AP$$
be the orbit map.
Let $P(t)$ be a smooth curve in $U(4, \mathbb{F})$ such that $P(0)=I_4$ and
$P'(0)=X\in T_{I_4}U(4, \mathbb{F})$.
Then we have
$$d\phi_{I_4}(X)=\frac{d}{dt}\bigg|_{t=0}\phi(P(t))
=\frac{d}{dt}\bigg|_{t=0}P^*(t)AP(t)=X^*A+AX=AX-XA.$$
By direct computations, we have
$$d\phi_{I_4}(\check{E}_{ij})=\left\{\begin{aligned}
	\hat{E}_{ij}\ \ \ &\mbox{if}\ (i, j)=(1, 3), (1, 4), (2, 3), (2, 4);\\
	0\ \ \ \ \ &\mbox{if}\ (i, j)=(1, 2), (3, 4).
\end{aligned}\right.$$
For $\mathbb{F}=\mathbb{C}$, we have
$$d\phi_{I_n}(\mathbf{i}\hat{E}_{ij})=\left\{\begin{aligned}
	\mathbf{i}\check{E}_{ij}\ \ \ &\mbox{if}\ (i, j)=(1, 3), (1, 4), (2, 3), (2, 4);\\
	0\ \ \ \ \ &\mbox{if}\ (i, j)=(1, 2), (3, 4),
\end{aligned}\right.$$
and $$d\phi_{I_n}(\mathbf{i}E_{ii})=0.$$
For $\mathbb{F}=\mathbb{H}$, we also have
$$d\phi_{I_n}(\mathbf{q}\hat{E}_{ij})=\left\{\begin{aligned}
	\mathbf{q}\check{E}_{ij}\ \ \ &\mbox{if}\ (i, j)=(1, 3), (1, 4), (2, 3), (2, 4);\\
	0\ \ \ \ \ &\mbox{if}\ (i, j)=(1, 2), (3, 4),
\end{aligned}\right.$$
and $$d\phi_{I_n}(\mathbf{q}E_{ii})=0\
\mbox{for}\ \mathbf{q}\in\{\mathbf{i}, \mathbf{j}, \mathbf{k}\}.$$
Note that $\phi$ is a smooth submersion. It follows that
$$T_A\mathcal{C}_{4, \mathbb{R}}=\operatorname{Span}
\big\{\hat{E}_{13}, \hat{E}_{14}, \hat{E}_{23}, \hat{E}_{24}\big\},$$
$$T_A\mathcal{C}_{4, \mathbb{C}}=\operatorname{Span}
\big\{\hat{E}_{13}, \hat{E}_{14}, \hat{E}_{23}, \hat{E}_{24}\big\}
\cup\big\{\mathbf{i}\check{E}_{13}, \mathbf{i}\check{E}_{14},
\mathbf{i}\check{E}_{23}, \mathbf{i}\check{E}_{24}\big\}$$
and
$$T_A\mathcal{C}_{4, \mathbb{H}}=\operatorname{Span}
\big\{\hat{E}_{13}, \hat{E}_{14}, \hat{E}_{23}, \hat{E}_{24}\big\}
\cup\big\{\mathbf{q}\check{E}_{13}, \mathbf{q}\check{E}_{14},
\mathbf{q}\check{E}_{23}, \mathbf{q}\check{E}_{24}:
\mathbf{q}=\mathbf{i}, \mathbf{j}, \mathbf{k}\big\}.$$
Moreover, the corresponding normal spaces are
$$N_A\mathcal{C}_{4, \mathbb{R}}=\operatorname{Span}
\big\{\hat{E}_{12}, \hat{E}_{34}, \eta_1, \eta_2\big\},$$
$$N_A\mathcal{C}_{4, \mathbb{C}}=\operatorname{Span}
\big\{\hat{E}_{12}, \hat{E}_{34}, \eta_1, \eta_2\big\}
\cup\big\{\mathbf{i}\check{E}_{12}, \mathbf{i}\check{E}_{34}\big\}$$
and
$$N_A\mathcal{C}_{4, \mathbb{H}}=\operatorname{Span}
\big\{\hat{E}_{12}, \hat{E}_{34}, \eta_1, \eta_2\big\}
\cup\big\{\mathbf{q}\check{E}_{12}, \mathbf{q}\check{E}_{34}:
\mathbf{q}=\mathbf{i}, \mathbf{j}, \mathbf{k}\big\},$$
where $\eta_1=\frac1{\sqrt{2}}\(E_{11}-E_{22}\), \eta_2=\frac1{\sqrt{2}}\(E_{33}-E_{44}\)$. Hence,  we can write any normal vector $\xi\in N_A\mathcal{C}_{4, \mathbb{F}}$ as
$$\xi=\begin{pmatrix}
		c & a & 0 & 0\\
		\bar{a} & -c & 0 & 0\\
		0 & 0 & d & b\\
		0 & 0 & \bar{b} & -d\\
	\end{pmatrix},$$
where $a,b\in\mathbb{F}$, $c,d\in\mathbb{R}$. Denote by $a_0:=\mathfrak{R}(a)$ and $b_0:=\mathfrak{R}(b)$. Define $P_{\xi}:=\operatorname{diag}(P_1,P_2)$ in the isotropy
group of the orbit $\mathcal{C}_{4, \mathbb{F}}$ at $A$ by
$$P_1:=\frac{1}{\sqrt{a_0^2+c^2}}\begin{pmatrix}
		a_0 & -c \\
		-c & -a_0\\
	\end{pmatrix},
\quad P_2:=\frac{1}{\sqrt{b_0^2+d^2}}\begin{pmatrix}
		b_0 & -d \\
		-d & -b_0\\
	\end{pmatrix},$$
where $P_1=\operatorname{diag}(1,-1)$ if $a_0^2+c^2=0$, and $P_2=\operatorname{diag}(1,-1)$ if $b_0^2+d^2=0$. Then it is easy to verify that the adjoint action of $P_\xi$ at $\xi \in N_A\mathcal{C}_{4, \mathbb{F}}$ satisfies $$P_\xi\cdot \xi =-\xi.$$
As the action is isometric, the discussions above show that the shape operators at $A$ in directions $\xi$ and $-\xi$ conjugate to each other, which by homogeneity implies the following.
\begin{prop}\label{Cn austere}
$\mathcal{C}_{4, \mathbb{F}}$ is a $4m$-dimensional austere submanifold of $S(4, \mathbb{F})$.
\end{prop}

 In order to compare with Bryant's constructions of austere subspaces, we provide explicit computations of shape operators of $\mathcal{C}_{4, \mathbb{F}}$ at $A$ for $\mathbb{F}=\mathbb{R},\mathbb{C}$ as follows.

	For each $\hat{E}_{ij}\in T_A\mathcal{C}_{4, \mathbb{R}}$, let
	$$\gamma_{\hat{E}_{ij}}(t)=\exp(-t\check{E}_{ij})A\exp(t\check{E}_{ij}).$$
	Then we have $\gamma_{\hat{E}_{ij}}(0)=A$ and $\gamma_{\hat{E}_{ij}}'(0)=\hat{E}_{ij}$.
	For sufficiently small $\varepsilon>0$,
	$\left.\gamma_{\hat{E}_{ij}}\right|_{(-\varepsilon, \varepsilon)}$ is a smooth embedding.
	For any $\xi\in N_A\mathcal{C}_{4, \mathbb{F}}$, let
	$$\hat{\xi}(\gamma_{\hat{E}_{ij}}(t))=\exp(-t\check{E}_{ij})\xi\exp(t\check{E}_{ij}).$$
	Then $\hat{\xi}$ is a smooth section of the restriction of the normal bundle $N\mathcal{C}_{4, \mathbb{F}}$
	to the embedded submanifold $\gamma_{\hat{E}_{ij}}(-\varepsilon, \varepsilon)$.
	By the extension lemma for sections of restricted bundles \cite[Problem 10-9]{Lee},
	there exists a smooth extension $\tilde{\xi}$ of $\hat{\xi}$ on a neighborhood of $\gamma_{\hat{E}_{ij}}(-\varepsilon, \varepsilon)$.
	Hence we have
	$$\begin{aligned}
		S_{\xi}(\hat{E}_{ij})
		&=-\(\left.\frac{d}{dt}\right|_{t=0}\tilde{\xi}(\gamma_{\hat{E}_{ij}}(t))\)^{\top}\\
		&=-\(\left.\frac{d}{dt}\right|_{t=0}\exp(-t\check{E}_{ij})\xi\exp(t\check{E}_{ij})\)^{\top}\\
		&=-\(\xi\check{E}_{ij}-\check{E}_{ij}\xi\)^{\top}.
	\end{aligned}$$
	
	First, we assume $\mathbb{F}=\mathbb{R}$.
	Then by direct computations, the matrix representations of
	$S_{\hat{E}_{12}}$, $S_{\hat{E}_{34}}$, $S_{\eta_1}$, $S_{\eta_2}$
	with respect to the basis $\big\{\hat{E}_{13}, \hat{E}_{14}, \hat{E}_{23}, \hat{E}_{24}\big\}$ are
	$$A_1=\frac{-1}{\sqrt{2}}\begin{pmatrix}
		0 & 0 & 1 & 0\\
		0 & 0 & 0 & 1\\
		1 & 0 & 0 & 0\\
		0 & 1 & 0 & 0\\
	\end{pmatrix},
	A_2=\frac{1}{\sqrt{2}}\begin{pmatrix}
		0 & 1 & 0 & 0\\
		1 & 0 & 0 & 0\\
		0 & 0 & 0 & 1\\
		0 & 0 & 1 & 0\\
	\end{pmatrix},$$
	$$A_3=\frac{-1}{\sqrt{2}}\begin{pmatrix}
		1 & 0 & 0 & 0\\
		0 & 1 & 0 & 0\\
		0 & 0 & -1 & 0\\
		0 & 0 & 0 & -1\\
	\end{pmatrix},
	A_4=\frac1{\sqrt{2}}\begin{pmatrix}
		1 & 0 & 0 & 0\\
		0 & -1 & 0 & 0\\
		0 & 0 & 1 & 0\\
		0 & 0 & 0 & -1\\
	\end{pmatrix},$$
	respectively.
	Let $V_1=\operatorname{Span}\big\{A_1, A_2, A_3, A_4\}$, and let
	$$P=\frac1{\sqrt{2}}\begin{pmatrix}
		1 & 0 & 0 & 1\\
		0 & -1 & 1 & 0\\
		0 & 1 & 1 & 0\\
		1 & 0 & 0 & -1\\
	\end{pmatrix}\in\operatorname{SO}(4).$$
	Since
	$$P^*A_1P=\frac{-1}{\sqrt{2}}\begin{pmatrix}
		0 & 0 & 1 & 0\\
		0 & 0 & 0 & 1\\
		1 & 0 & 0 & 0\\
		0 & 1 & 0 & 0\\
	\end{pmatrix},\
	P^*A_2P=\frac{1}{\sqrt{2}}\begin{pmatrix}
		0 & 0 & 1 & 0\\
		0 & 0 & 0 & -1\\
		1 & 0 & 0 & 0\\
		0 & -1 & 0 & 0\\
	\end{pmatrix},$$
	$$P^*A_3P=\frac{-1}{\sqrt{2}}\begin{pmatrix}
		0 & 0 & 0 & 1\\
		0 & 0 & -1 & 0\\
		0 & -1 & 0 & 0\\
		1 & 0 & 0 & 0\\
	\end{pmatrix},\
	P^*A_4P=\frac1{\sqrt{2}}\begin{pmatrix}
		0 & 0 & 0 & 1\\
		0 & 0 & 1 & 0\\
		0 & 1 & 0 & 0\\
		1 & 0 & 0 & 0\\
	\end{pmatrix},$$
	every matrix in the space $P^*V_1P$ is of the form
	$$\begin{pmatrix}
		0 & 0 & a & c\\
		0 & 0 & d & b\\
		a & d & 0 & 0\\
		c & b & 0 & 0\\
	\end{pmatrix},$$
	where $a, b, c, d\in\mathbb{R}$.
	From the constructions in Table \ref{max_austere}, we know that
	$P^*V_1P$ is a $4$-dimensional austere subspace,
	and hence $\mathcal{C}_{4, \mathbb{R}}$ is an austere submanifold.
	
	Next, we assume $\mathbb{F}=\mathbb{C}$.
	For each $\mathbf{i}\check{E}_{ij}\in T_A\mathcal{C}_{4, \mathbb{C}}$, let
	$$\gamma_{\mathbf{i}\check{E}_{ij}}(t)=\exp(-t\mathbf{i}\hat{E}_{ij})A\exp(t\mathbf{i}\hat{E}_{ij}).$$
	Then we have $\gamma_{\mathbf{i}\check{E}_{ij}}(0)=A$ and $\gamma_{\mathbf{i}\check{E}_{ij}}'(0)=\mathbf{i}\check{E}_{ij}$.
	By similar computation, the matrix representations of
	$S_{\hat{E}_{12}}$, $S_{\hat{E}_{34}}$, $S_{\eta_1}$, $S_{\eta_2}$,
	$S_{\mathbf{i}\check{E}_{12}}$, $S_{\mathbf{i}\check{E}_{34}}$
	with respect to the basis $\big\{\hat{E}_{13}, \hat{E}_{14}, \hat{E}_{23}, \hat{E}_{24}\big\}
	\cup\big\{\mathbf{i}\check{E}_{13}, \mathbf{i}\check{E}_{14},
	\mathbf{i}\check{E}_{23}, \mathbf{i}\check{E}_{24}\big\}$ are
	$$B_1=\begin{pmatrix}
		A_1 & 0\\
		0 & A_1\\
	\end{pmatrix},
	B_2=\begin{pmatrix}
		A_2 & 0\\
		0 & A_2\\
	\end{pmatrix},
	B_3=\begin{pmatrix}
		A_3 & 0\\
		0 & A_3\\
	\end{pmatrix},
	B_4=\begin{pmatrix}
		A_4 & 0\\
		0 & A_4\\
	\end{pmatrix},$$
	$$B_5=\begin{pmatrix}
		0 & A_5\\
		-A_5 & 0\\
	\end{pmatrix},
	B_6=\begin{pmatrix}
		0 & A_6\\
		-A_6 & 0\\
	\end{pmatrix},$$
	respectively, where
	$$A_5=\frac{1}{\sqrt{2}}\begin{pmatrix}
		0 & 0 & 1 & 0\\
		0 & 0 & 0 & 1\\
		-1 & 0 & 0 & 0\\
		0 & -1 & 0 & 0\\
	\end{pmatrix},\
	A_6=\frac{1}{\sqrt{2}}\begin{pmatrix}
		0 & 1 & 0 & 0\\
		-1 & 0 & 0 & 0\\
		0 & 0 & 0 & 1\\
		0 & 0 & -1 & 0\\
	\end{pmatrix}.$$
	Let $V_2=\operatorname{Span}\big\{B_1, B_2, B_3, B_4, B_5, B_6\}$, and let
	$$Q=\begin{pmatrix}
		0 & 0 & I_2 & 0\\
		0 & I_2 & 0 & 0\\
		I_2 & 0 & 0 & 0\\
		0 & 0 & 0 & I_2\\
	\end{pmatrix}
	\begin{pmatrix}
		P & 0\\
		0 & P\\
	\end{pmatrix}\in\operatorname{SO}(8).$$
	Since
	$$P^*A_5P=\frac{1}{\sqrt{2}}\begin{pmatrix}
		0 & 1 & 0 & 0\\
		-1 & 0 & 0 & 0\\
		0 & 0 & 0 & -1\\
		0 & 0 & 1 & 0\\
	\end{pmatrix},\
	P^*A_6P=\frac{-1}{\sqrt{2}}\begin{pmatrix}
		0 & 1 & 0 & 0\\
		-1 & 0 & 0 & 0\\
		0 & 0 & 0 & 1\\
		0 & 0 & -1 & 0\\
	\end{pmatrix},$$
	direct computations imply that every matrix in the space $Q^*V_2Q$ is of the form
	$$\begin{pmatrix}
		0 & 0 & 0 & 0 & 0 & e & a & c\\
		0 & 0 & 0 & 0 & -e & 0 & d & b\\
		0 & 0 & 0 & 0 & a & d & 0 & f\\
		0 & 0 & 0 & 0 & c & b & -f & 0\\
		0 & -e & a & c & 0 & 0 & 0 & 0\\
		e & 0 & d & b & 0 & 0 & 0 & 0\\
		a & d & 0 & -f & 0 & 0 & 0 & 0\\
		c & b & f & 0 & 0 & 0 & 0 & 0\\
	\end{pmatrix},$$
	where $a, b, c, d, e, f\in\mathbb{R}$.
	From the constructions in Table \ref{max_austere}, we know that
	$Q^*V_2Q$ is a $6$-dimensional austere subspace,
	and hence $\mathcal{C}_{4, \mathbb{C}}$ is an austere submanifold.

\section{Dimension upper bound of austere subspaces}\label{sec-dim}
We will use the submanifold structure of $\mathcal{B}_{n, \mathbb{R}}$ to
estimate the maximal dimension of austere subspaces which have nonempty intersection with $\mathcal{B}_{n, \mathbb{R}}$.
We start with the following preparations.
\begin{defn}\cite[p.11]{Scharlau}
Let $V$ be a real linear space,
and let $b$ be a symmetric bilinear form on $V$.
A subspace $W$ is called a totally isotropic subspace of $(V, b)$ if $b(W, W)=0$.
\end{defn}

\begin{lem}\label{ind_est}
Let $V$ be an $n$-dimensional real linear space,
and let $b$ be a symmetric bilinear form of signature $(r, s)$ on $V$.
If $W$ is a totally isotropic subspace of $(V, b)$, then
$$\operatorname{dim}W\leq (n-r-s)+\min\{r, s\}.$$
\end{lem}
\begin{proof}
Without loss of generality, we can assume $r\geq s$.
Then we need to show that $\operatorname{dim}W\leq n-r$.
By the knowledge of linear algebra,
we can choose a $r$-dimensional linear subspace $V_1$ of $V$ such that
$b$ is positive definite on $V_1$.
Let $\pi$ be the $b$-orthogonal projection of $V$ onto $V_1^{\bot}$.
Since $W$ is totally isotropic,
we have $$W\cap\operatorname{Ker}\pi=W\cap V_1=\{0\}.$$
Hence $$\operatorname{dim}W=\operatorname{dim}\pi(W)
\leq\operatorname{dim}V_1^{\bot}=n-r.$$
\end{proof}
	
\begin{rem}
When $b$ is nondegenerate, the maximal dimension of totally isotropic
subspaces is called the Witt index \cite[p.17]{Scharlau} or the index of isotropy \cite[p.57]{MH73}.
If the nondegeneracy is assumed,
our Lemma \ref{ind_est} is a special case of \cite[p.57, Main lemma]{MH73}.
\end{rem}

\begin{thm}\label{dimest}
Let $\mathcal{Q}$ be an austere subspace in $M(n, \mathbb{R})$ such that
$\mathcal{Q}\cap\mathcal{B}_{n, \mathbb{R}}\neq\varnothing$, then
\begin{equation}\label{dimQ}
\operatorname{dim}\mathcal{Q}\leq\left\{\begin{aligned}
p^2+2p\ \ \ \ \ &\mbox{if}\ n=2p+1;\\
p^2+3p+2\ \ \ &\mbox{if}\ n=2p+2.
\end{aligned}\right.
\end{equation}
\end{thm}
\begin{proof}
Since $\mathcal{Q}$ is an austere subspace, we know that
$\mathcal{Q}\cap S(n, \mathbb{R})$ is a totally geodesic sphere and $$\mathcal{Q}\cap S(n, \mathbb{R})\subset\mathcal{A}_n
=\mathcal{B}_{n, \mathbb{R}}\cup\mathcal{C}_{n, \mathbb{R}}.$$
Choose a point $p\in \mathcal{Q}\cap\mathcal{B}_{n, \mathbb{R}}$.
Since $\mathcal{C}_{n, \mathbb{R}}$ is closed in $S(n, \mathbb{R})$, there exists an open neighborhood $U$ of $p$ in $S(n, \mathbb{R})$ such that $U\cap\mathcal{C}_{n, \mathbb{R}}=\varnothing$.
Thus $V=U\cap\mathcal{Q}$ is an open submanifold of $\mathcal{Q}\cap S(n, \mathbb{R})$,
and $V\subset U\cap\mathcal{A}_n=U\cap\mathcal{B}_{n, \mathbb{R}}$.
Therefore, $V$ is a totally geodesic embedded submanifold of $\mathcal{B}_{n, \mathbb{R}}$, and
\begin{equation}\label{dimVQ}
\operatorname{dim}V=\operatorname{dim}\mathcal{Q}\cap S(n, \mathbb{R})=\operatorname{dim}\mathcal{Q}-1.
\end{equation}
We can characterize $T_pV$ as a subspace of $T_p\mathcal{B}_{n, \mathbb{R}}$, and
we recall that $$\operatorname{dim}T_p\mathcal{B}_{n, \mathbb{R}}
=N(n, \mathbb{R})-1-p=\left\{\begin{aligned}
2p^2+2p-1\ \ \ &\mbox{if}\ n=2p+1;\\
2p^2+4p+1\ \ \ &\mbox{if}\ n=2p+2.
\end{aligned}\right.$$
The second fundamental form $\operatorname{II}_{\xi_1}$ of $\mathcal{B}_{n, \mathbb{R}}$ in the direction $\xi_1$ is a symmetric bilinear form on $T_p\mathcal{B}_{n, \mathbb{R}}$.
By the proof of Proposition \ref{Bn austere}, we know that the nullity of $\operatorname{II}_{\xi_1}$ is
$$\operatorname{Nul}\(\operatorname{II}_{\xi_1}\)=\left\{\begin{aligned}
2p-1\ \ \ &\mbox{if}\ n=2p+1;\\
2p+1\ \ \ &\mbox{if}\ n=2p+2,
\end{aligned}\right.$$
the index of $\operatorname{II}_{\xi_1}$ is $$\operatorname{Ind}\(\operatorname{II}_{\xi_1}\)=\left\{\begin{aligned}
p^2\ \ \ \ \ &\mbox{if}\ n=2p+1;\\
p^2+p\ \ \ &\mbox{if}\ n=2p+2,
\end{aligned}\right.$$
and the signature of $\operatorname{II}_{\xi_1}$ is
$\(\operatorname{Ind}\(\operatorname{II}_{\xi_1}\), \operatorname{Ind}\(\operatorname{II}_{\xi_1}\)\)$.
Since $V$ is totally geodesic in $S(n, \mathbb{R})$, we have
$$\operatorname{II}_{\xi_1}(X, Y)=0$$ for any $X, Y\in T_pV$, i.e.,
$T_pV$ is a totally isotropic subspace of $(T_p\mathcal{B}_{n, \mathbb{R}}, \operatorname{II}_{\xi_1})$.
Then by Lemma \ref{ind_est},
\begin{equation}\label{dimV}
\operatorname{dim}T_pV\leq\operatorname{Nul}\(\operatorname{II}_{\xi_1}\)+\operatorname{Ind}\(\operatorname{II}_{\xi_1}\)
=\left\{\begin{aligned}
p^2+2p-1\ \ \ &\mbox{if}\ n=2p+1;\\
p^2+3p+1\ \ \ &\mbox{if}\ n=2p+2.
\end{aligned}\right.
\end{equation}
Finally, (\ref{dimQ}) follows from (\ref{dimVQ}) and (\ref{dimV}).
\end{proof}

\begin{rem}
If $n$ is even, then the equality in (\ref{dimQ}) is achieved by one of Bryant's constructions in Table \ref{max_austere}.
The intersection condition in Theorem \ref{dimest} seems to be removable.
\end{rem}

\begin{ack}
The first author sincerely thank Professor Gudlaugur Thorbergsson for his detailed and patient guidance on Dupin hypersurfaces when Ge was a Humboldt fellow hosted by him in the university of Cologne between 2012 and 2014. Ge is also indebted to Professor Tom Cecil for sending his excellent book \emph{Lie Sphere Geometry} during the workshop at Tohoku university in 2010. The authors also thank Gudlaugur and Tom for their helpful discussions and insightful comments which make this paper clearer and more concrete.
At last, the authors would like to thank Professor Robert L. Bryant for his enlightening discussions during the 2023 Symposium on Geometry-100th anniversary of Professor Eugenio Calabi.
\end{ack}


\end{document}